\documentclass[preprint,12pt]{elsarticle}
\usepackage{amsfonts,color}
\usepackage{graphicx}
\usepackage[text={6.5in,9.25in},centering]{geometry}
\usepackage[colorlinks=true,urlcolor=blue,citecolor=blue]{hyperref}
\usepackage[utf8]{inputenc}
\usepackage{bm}
\usepackage{amsmath}
\usepackage{enumitem}
\usepackage{amsmath,amsfonts,amssymb,amsthm,mathtools}
\usepackage{graphicx}
\usepackage{algorithm}
\usepackage{algpseudocode}
\usepackage{setspace}
\usepackage{booktabs} 
\usepackage{siunitx} 
\usepackage{hyperref}
\usepackage{comment}

\usepackage{natbib} % or biblatex if preferred
\newtheorem{theorem}{Theorem}

\newtheorem{remark}{Remark}
\begin{document}

%\makeatletter
%\def\ps@pprintTitle{%
  %\let\@oddhead\@empty
  %\let\@evenhead\@empty
  %\let\@oddfoot\@empty
  %\let\@evenfoot\@empty}
%\makeatother

\begin{frontmatter}

\title{Data-Driven Parameter Identification for Tumor Growth Models}

\author[1]{Liu Liu}
\author[2]{Yifei Wang}
\author[2]{Qinyu Xu}
\author[2]{Xiaoqian Xu}

\affiliation[1]{organization={The Chinese University of Hong Kong},
               addressline={Department of Mathematics},
               city={Shatin},
               postcode={999077},
               state={New Territories},
               country={Hong Kong, China}}

\affiliation[2]{organization={Duke Kunshan University},
               addressline={Zu Chongzhi Center}, 
               city={Kunshan},
               postcode={215316}, 
               state={Jiangsu},
               country={China}}

\begin{abstract}
Modeling tumor growth accurately is essential for understanding cancer progression and informing treatment strategies. To estimate the parameters in the tumor growth model described by a nonlinear PDE, we adopt Physics-Informed Neural Networks (PINNs) \cite{CLKD2020} and DeepONet \cite{lu2021learning}, which show advantages especially when the observation data is scarce and contains noise. With the help of real-life lab data, we have demonstrated the potential of applying deep learning tools to address data-driven modeling for tumor growth in biology.
%In this study, we use Physics-Informed Neural Networks (PINNs) to estimate tumor growth parameters in a nonlinear partial differential equation (PDE) model describing tumor cell density over space and time. PINNs are trained by integrating physical laws with observed data. Results show that it can recover the unknown parameters and reproduce tumor growth dynamics consistent with unseen
%data, demonstrating its potential as a powerful tool for data-driven biological modeling.
\end{abstract}

\end{frontmatter}

\section{Introduction}
Understanding and predicting the behavior of solid tumors is a central goal in mathematical biology. Accurate modeling of tumor growth is not only essential for understanding cancer progression, but also plays a vital role in designing treatment strategies, optimizing intervention timing, and predicting patient-specific outcomes. In the early stages of tumor development, before angiogenesis occurs, tumor growth is primarily driven by the diffusion of nutrients and waste products through surrounding tissues. Classical models, such as those proposed by Greenspan~\cite{greenspan1972models}, describe how diffusion-limited growth can lead to characteristic tumor structures, including a necrotic core surrounded by a proliferating rim. For the more modern discussion about the tumor model, one can check \cite{cristini2003nonlinear,friedman2001symmetry,cristini2010multiscale,cristini2017introduction,araujo2004history,byrne2006modelling,roose2007mathematical,lowengrub2009nonlinear,perthame2016some}.
%As the tumor enlarges and nutrient transport becomes insufficient in the interior, cell death can occur, giving rise to a necrotic region. However, in this research, we focus on earlier stages of tumor development where necrosis is negligible or absent, and thus the formation of a necrotic core is not considered.

To capture the nonlinear and spatially heterogeneous nature of tumor growth, modern approaches often involve nonlinear partial differential equations (PDEs) incorporating density-dependent diffusion and proliferation. A central component of such models is the proliferation coefficient, which governs how quickly tumor cells grow in response to available resources. Accurately estimating this coefficient (or coefficients) is critical, as it directly relates to the tumor's intrinsic growth rate and its potential response to treatment. However, these parameters are rarely directly measurable in experimental or clinical settings, making their inference from partial and noisy observations a challenging but necessary task.

Previous studies have successfully employed statistical and computational inversion techniques such as Bayesian inference~\cite{falco2023quantifying,feng2024unified} to estimate biological problems like tumor growth parameters, providing valuable probabilistic insights and a principled framework for quantifying uncertainty in parameter estimation. % Building on the success of these approaches, 
In this study we explore an alternative method, Physics-Informed Neural Networks (PINNs)\cite{raissi2019physics}, which integrates physical laws directly into the training of neural networks to enable efficient and robust parameter estimation, even when the available data are sparse, noisy, or partially observed \cite{wang2024neural}.

%Machine Learning, Deep neural networks (DNNs) in particular, have gained increasing interest in approximating the solutions of partial differential equations (PDE) due to their universal approximation property and ability to handle high-dimensional problems. PINNs are a class of deep learning methods that incorporate physical laws, expressed as PDEs, directly into the training process of the neural networks \cite{raissi2019physics}. 

%Unlike conventional neural networks that rely solely on data, PINNs are trained to satisfy both the observed data and the underlying governing equations, allowing them to infer unknown solutions or parameters with greater physical consistency. 
PINNs have shown remarkable promise in various applications, including solving both forward and inverse problems \cite{CLKD2020, MPG2019,chen2020physics,lou2021physics,mao2020physics}. Recent studies have demonstrated the effectiveness of PINNs in parameter estimation across diverse biomedical contexts, including blood flow dynamics \cite{kissas2020machine} and cardiac electrophysiology \cite{sahli2020physics}, underscoring its potential as a powerful tool for data-driven discovery and quantitative prediction in complex biomedical systems. Furthermore, this method offers major advantages for tumor growth modeling: by leveraging the underlying physics it allows the model to generalize well even with limited data, and enables robust inference of unknown parameters based on the observed spatiotemporal tumor density. In this work, we aim to validate and apply PINNs frameworks to study parameters prediction for tumor growth models, building on its theoretical convergence properties.
Complementary to PINNs, Deep Operator Networks (DeepONets) \cite{lu2021learning} provide a distinct but related framework for scientific machine learning. Rather than solving a single PDE instance, DeepONets learn the solution operator that maps input functions (such as initial conditions or spatially varying parameters) to output functions, enabling rapid evaluation across a family of PDE problems. This operator-learning perspective makes DeepONets particularly well-suited for scenarios where multiple parameter configurations must be explored, or where the initial tumor density is unknown and treated as a functional input. In this work, we aim to validate and apply both the PINNs and DeepONet frameworks to study parameter prediction for tumor growth models, building on their respective theoretical convergence properties.
%solve the inverse problem of parameter estimation in tumor growth models. 
Before applying the PINNs framework to real observations, we first validate its reliability through a two-step verification process:

\textbf{Step 1: Recovering the proliferation rate from synthetic tumor growth data.} 
We generate synthetic tumor growth data by using an efficient numerical scheme developed in \cite{liu2018accurate} to solve the underlying PDE 
for a range of proliferation rates $v$. 
The PINNs model is trained to recover these parameters from the simulated spatiotemporal tumor density data, then we compare the approximated values with the ground truth to assess accuracy.

\textbf{Step 2: Robustness to noisy data.} 
To emulate real-world experimental conditions, controlled Gaussian noise is added to the synthetic numerical data. The PINNs framework is retrained on these noisy datasets, and the resulting parameter predictions are analyzed to evaluate convergence behavior and sensitivity to different noise levels. This step not only tests the noise tolerance of the framework, but it also justifies its capability to real experimental data.

Upon validation of the framework's performance on synthetic tumor growth data and its robustness to noisy data, we further apply it to actual observed tumor growth data, demonstrating its practical capability under real-world experimental conditions. We make the biologically reasonable assumption of radial symmetry in tumor growth, with a radially symmetric initial condition and zero Dirichlet boundary conditions at domain edges. The training strategy leverages limited observed data, by using a separate testing set for validation the results show that the predicted tumor growth patterns match the testing data well with relative errors below 5\% in tumor radius prediction. Furthermore, we extend the framework to more complex inverse problem scenarios, including various sources of parameters such as spatially varying proliferation rates and initial conditions. 

%This study establishes PINNs as a powerful computational framework for addressing inverse problems in tumor growth modeling, with particular strength in handling data-scarce and noisy environments typical of biomedical applications. By providing a robust methodology for inferring biologically critical parameters from limited observations, our work enables more accurate characterization of tumor progression dynamics. This capability opens new avenues for understanding cancer mechanisms and contributes to the development of personalized treatment strategies, ultimately advancing precision oncology through improved model-based prediction of therapeutic outcomes. 

%This study shows that PINNs offer a powerful computational framework to address inverse problems in tumor growth modeling. Supported by both theoretical convergence guarantees and validation for test case using synthetic data, the framework shows advantages in handling data-scarce and noisy environment which is challenging yet typical in biomedical applications. The methodology provides a reliable approach for parameter estimation in tumor growth models, providing new pathways to push further study in this field with real-life applications in the lab. 
While PINNs and DeepONet frameworks have been widely explored, their application to tumor growth modeling--particularly within the framework of nonlinear Porous Medium Equations (PME)--remains challenging due to the lack of density data in clinical settings. We emphasize that many of the literature on this topic uses synthetic data obtained from traditional numerical solvers of the underlying PDEs, whereas we use the real experimental data from the biology lab. The main contribution of our work lie in the following: (i) establishing a robust PINNs and DeepONet framework for PME-based tumor models and study its inverse problems; (ii) introducing a binary-label-informed loss function (see equation (27)) that bridges the gap between raw experimental image data and quantitative parameter identification; (iii) providing theoretical convergence analysis for the proposed neural network approach, which justifies its use in the complex, data-scarce and noisy environments typical of biological research. We provide reliable approaches for parameter estimation in tumor growth models, providing new pathways to push further  this field with real-life applications. To the best of our knowledge, it is rare in the literature to present both experimental lab data and theoretical convergence analysis within the same study. By enabling accurate parameter identification from binary-labeled images and scarce or noisy experimental data, this work offers a data-driven approach to track tumor progression more objectively. We believe the results achieved in this work can serve as a reproducible and robust computational tool for the applied mathematics and oncology community, providing a foundation to help and guide clinical decision-making in reality.

The rest of the paper is organized as follows. Section~\ref{Modeling of Tumor Growth} introduces the mathematical formulation of the tumor growth model, including the governing porous medium equation, initial–boundary conditions, and modeling assumptions. Section~\ref{Methods} describes the PINNs framework used throughout this work, detailing the network architecture, construction of the loss function, and the overall training procedure for solving inverse problems. Section~\ref{Analysis} provides theoretical analysis for the proposed approach, including the convergence properties of both the loss function and the neural network solution. In Section~\ref{Preliminary}, we present preliminary validation based on synthetic tumor density data generated from numerical solvers, examining parameter recovery and further validating robustness in the presence of noise. Section~\ref{Prediction of Proliferation Rate Based on Observed Data} applies the framework to experimentally observed tumor growth data using binary-labeled measurements, then assessing its predictive performance by comparing the predicted tumor radii with testing data to quantify the radius prediction error. In Section~\ref{Application to Multiple Unknown Parameters}, the PINNs framework is applied to more complex scenarios involving spatially varying proliferation rates and unknown initial tumor density. Finally, in Section~\ref{Conclusion} we conclude the study and outline possible future research. 

\section{Modeling of tumor growth}
\label{Modeling of Tumor Growth}
We introduce the tumor growth model which is described by porous medium equation (\cite{perthame2016some}): 
\begin{equation}
\begin{aligned}
\rho_t + \nabla \cdot (\rho \mathbf{v}) &= \mathbf{g}(x,y,t) \rho, \\
\rho(x,y, 0) &= \rho_0(x,y), 
\end{aligned}
\label{eq: PDEs}
\end{equation}
where $\rho(x,y,t)$ presents the tumor cell density, $\mathbf{v}$ is the velocity of cells, and $g(x,y,t)$ is the proliferation rate of cells. According to the Darcy’s law:
\begin{equation}
\begin{aligned}
\mathbf{v} = -\nabla p,
\end{aligned}
\label{eq: Darcy}
\end{equation}
$p$ denotes the mechanical pressure generated by the crowded tumor cells, which is given by:
\begin{equation}
\begin{aligned}
p = P_m(\rho) := \frac{m}{m - 1} \rho^{m - 1}, \quad m \geq 2,
\end{aligned}
\label{eq: pressure}
\end{equation}

$m$ denotes the nonlinearity exponent in the constitutive relation between cell density and pressure, controlling how strongly the tumor pressure increases with cell density. Such density-dependent diffusion models and their connections to Darcy-type flow and free-boundary tumor growth dynamics have been systematically analyzed in Liu et al.~\cite{liu2019analysis}, providing a theoretical foundation for porous-medium formulations of tumor evolution.

Then we can get a family of models \(\{(P_m)\}_{m=2}^\infty\) with similar structure but different constitutive relation:

\begin{equation}
\begin{aligned}
\begin{cases} 
\rho_{t} - \Delta \rho^{m} = \mathbf{g}(\mathrm{x},\mathrm{y}, t) \rho, \\ 
\rho(\mathrm{x},\mathrm{y}, 0) = \rho_{0}(\mathrm{x},\mathrm{y}). 
\end{cases}
\end{aligned}
\label{eq: m_PDE}
\end{equation}
$\mathbf{g}(\mathrm{x},\mathrm{y}, t)$ in \eqref{eq: m_PDE} refers to the proliferation rate related to time and space which we want to predict. 

To begin with, we assume $\mathbf{g}(\mathrm{x},\mathrm{y}, t)$ is a constant called $v$. As mentioned in \cite{falco2023quantifying}, Falco et al. successfully employed Bayesian methods to predict tumor dynamics when choosing the parameter $m$ as 2, demonstrating that the Bayesian inference-based model can effectively capture the tumor growth dynamics. To further investigate the dynamics under different parameter regimes, we now set $m=3$ and utilize PINNs for prediction. 

Therefore, the tumor growth model that we address turns out to be:
\begin{equation}
\begin{aligned}
\rho_{t} - \Delta \rho^{3} = v \rho. 
\end{aligned}
\label{eq: 3_PDE}
\end{equation}
The initial tumor distribution is set to be a patch, such as:
\begin{equation}
\begin{aligned}
\rho_0(x,y) =
\begin{cases}
1, & \text{if } x^2 + y^2 < 0.25, \\
0, & \text{otherwise},
\end{cases}
\end{aligned}
\label{eq: initial}
\end{equation}

and homogeneous Dirichlet boundary conditions are imposed on all edges (i.e., $\rho = 0$ on $\partial\Omega$ for all $t$).

\section{Methods}
\label{Methods}
We consider a feedforward neural network with $L$ layers that maps the input 
$z^{(0)} = (t, x, y) \in \mathbb{R}^d$ (time and spatial coordinates) to the output 
$u_{\boldsymbol{\theta}}(t, x, y)$, which is the tumor density predicted by PINNs (here we use $u$ in our algorithm to replace $\rho$ in Equation \eqref{eq: m_PDE}). The layer-wise propagation is defined as follows \cite{liu2025asymptotic}:

\begin{equation}
\begin{aligned}
& z^{(l)} = \sigma \left( W^{(l)} z^{(l-1)} + b^{(l)} \right), \quad l = 1, 2, \ldots, L - 1, \\
& u_{\boldsymbol{\theta}}(t, x, y) = W^{(L)} z^{(L-1)} + b^{(L)}.
\end{aligned}
\label{eq: PINNs}
\end{equation}

Here:
\begin{itemize}
    \item $d$: dimension of the input space ($d=3$ for $(t,x,y)$),
    \item $z^{(0)} \in \mathbb{R}^d$: the input vector (time $t$ and spatial coordinates $x, y$),
    \item $n_l$: number of neurons in the $l$-th layer,
    \item $z^{(l)} \in \mathbb{R}^{n_l}$: the output of the $l$-th layer (hidden representation),
    \item $W^{(l)} \in \mathbb{R}^{n_l \times n_{l-1}}$: the weight matrix of the $l$-th layer,
    \item $b^{(l)} \in \mathbb{R}^{n_l}$: the bias vector of the $l$-th layer,
    \item $\sigma(\cdot)$: nonlinear activation function,
    \item $u_{\boldsymbol{\theta}}(t, x, y)$: the neural network output that approximates the PDE solution,
    \item $\boldsymbol{\theta} = \{ W^{(l)}, b^{(l)} \}_{l=1}^L$: the set of all trainable network parameters.
\end{itemize}

\begin{figure}[H]
\centering
\includegraphics[width=1\linewidth]{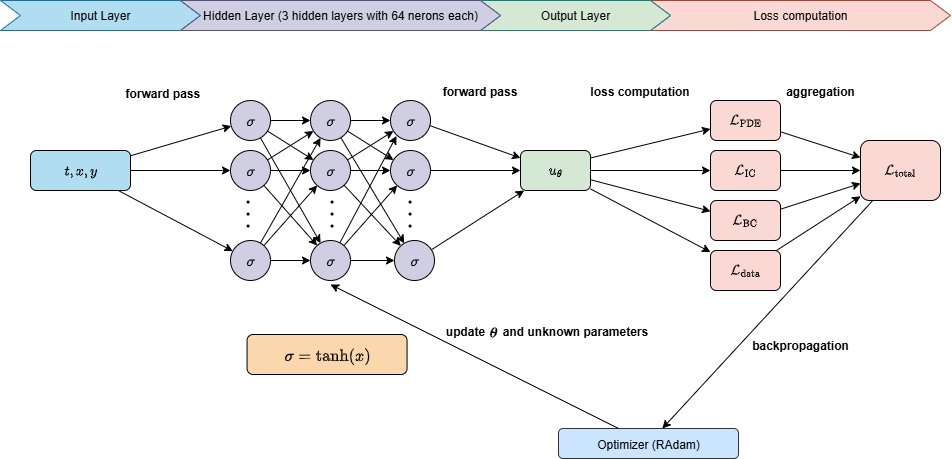} \\
\caption{PINNs setup and framework}
\label{framework}
\end{figure}
The basic PINNs setup and the composition of the loss function are depicted in Figure~\ref{framework}. In this work, the total loss function consists of four terms, one term related to data losses ($\mathcal{L}_{\text{data}}$) and three terms related to PDE losses ($\mathcal{L}_{\text{PDE}}, \mathcal{L}_{\text{IC}},\mathcal{L}_{\text{BC}}$), yielding the loss function with proper weights $w_1,w_2,w_3,w_4$:
\begin{equation}
\begin{aligned}
\mathcal{L}_{\text{total}} = w_1\mathcal{L}_{\text{PDE}} + w_2\mathcal{L}_{\text{IC}} + w_3\mathcal{L}_{\text{BC}} + w_4\mathcal{L}_{\text{data}}.
\end{aligned}
\label{eq: loss funtion}
\end{equation}
Typically, the losses are computed as the traditional mean-squared errors obtained after training the neural network:
\begin{equation}
\begin{aligned}
\mathcal{L}_{\text{PDE}} = \frac{1}{N} \sum [u_t - 6u(u_x^2+u_y^2)-3u^2(u_{xx} + u_{yy}) - vu]^2,
\end{aligned}
\label{eq: PDE loss}
\end{equation}

\begin{equation}
\begin{aligned}
\mathcal{L}_{\text{IC}} = \frac{1}{N_1} \sum (u - u_0)^2 |_{t=0},
\end{aligned}
\label{eq: IC loss}
\end{equation}

\begin{equation}
\begin{aligned}
\mathcal{L}_{\text{BC}} = \frac{1}{N_1} \sum (u - u_{\partial\Omega})^2 |_{x,y \in \partial \Omega},
\end{aligned}
\label{eq: BC loss}
\end{equation}

\begin{equation}
\begin{aligned}
\mathcal{L}_{\text{data}} = \frac{1}{N_2} \sum_i (u_i - \hat{u_i})^2,
\end{aligned}
\label{eq: data loss}
\end{equation}

To obtain the sampling points, we utilize torch.rand(n, 1) in PyTorch, which generates values uniformly distributed over \( [0, 1) \). In Equations \eqref{eq: PDE loss},\eqref{eq: IC loss},\eqref{eq: BC loss},\eqref{eq: data loss}, $\mathcal{L}_{\text{PDE}} $ quantifies how well the model satisfies the PDE governing tumor density evolution over $N$ collocation points in the spatiotemporal domain. $\mathcal{L}_{\text{IC}} $ enforces the initial conditions by penalizing deviations between the predicted density and the known initial state at time $t=0$ over $N_1$ points sampled from the initial condition. $\mathcal{L}_{\text{BC}}$ ensures the solution adheres to prescribed spatial constraints $u_{\partial \Omega}$ along the domain boundaries $\partial \Omega$. Like $\mathcal{L}_{\text{IC}} $, it is evaluated over $N_1$ boundary points. Lastly,  $\mathcal{L}_{\text{data}} $
quantifies agreement with experimental measurements by comparing the model’s predictions $\hat{u_i}$ against observed tumor densities $u_i$. This data fidelity term is averaged over $N_2$ measurement points.

Specifically, in this study, a total of $N = 2000$ collocation points are randomly sampled within the spatial-temporal domain $x \in [-3, 3]$, $y \in [-3, 3]$, $t \in [0, 1]$. For each of the four spatial boundaries (left, right, up, down), we sample $N_1 = 100$ points, resulting in a total of $4 \times 100 = 400$ boundary points. Additionally, $N_1 = 100$ points are sampled at the initial time $t = 0$, and we incorporate $N_2 = 200$ measurement data points from experimental or numerical simulations.

The model is implemented in  \href{https://colab.google/}{Google Colaboratory} to execute the code for training the PINNs and predicting the unknown parameters, and the general framework of solving the inverse problems is depicted in Algorithm~\ref{alg:pinn}.

\begin{algorithm}
\caption{\normalfont PINNs-based Inference of Unknown Parameters in PDEs}
\label{alg:pinn}
\begin{spacing}{1.2}
\begin{enumerate}
    \item Assign a proper initial guess to the unknown parameters and initialize the neural network $u_{\boldsymbol{\theta}}$.
    
    \item Extract a subset of data obtained from numerical simulations/real observations.
    
    \item Generate collocation points $(\boldsymbol{x},t)$ in the domain $\Omega \times (0,t)$, including initial, boundary, and interior points ($\boldsymbol{x} \in \mathbb{R}^n$).
    
    \item Compute the PINNs approximation $u = u_{\boldsymbol{\theta}}(\boldsymbol{x},t)$ and the required derivatives.
    
    \item Calculate the PDE residual
    \[
        \mathcal{R}(\boldsymbol{x}, t) = \partial_t u(\boldsymbol{x}, t) - \mathcal{N}[u(\boldsymbol{x}, t)]
    \]
    \item Calculate residuals for initial condition, boundary condition and evaluate data loss, then forming the total mean-squared error (MSE) loss by choosing appropriate values of weights $w_1,w_2,w_3,w_4$:
    \[
        \mathcal{L_{\text{total}}} = w_1\mathcal{L}_{\text{PDE}} + w_2\mathcal{L}_{\text{IC}} + w_3\mathcal{L}_{\text{BC}} + w_4\mathcal{L}_{\text{data}}.
    \]
    
    \item Update the values of unknown parameters and $\mathbf{\boldsymbol{\theta}}$ using RAdam to minimize $\mathcal{L_{\text{total}}}$.
    
    \item Repeat Steps 2–7 until convergence.
\end{enumerate}
\end{spacing}

Here, the loss components $\mathcal{L}_{\text{PDE}}$, $\mathcal{L}_{\text{IC}}$, $\mathcal{L}_{\text{BC}}$, and $\mathcal{L}_{\text{data}}$ enforce the PDE loss, initial condition loss, boundary condition loss, and data loss, respectively. $\boldsymbol{\theta}$ refers to the parameters within the neural network.
\end{algorithm}
\section{Analysis}
\label{Analysis}
\subsection{Classical theory of porous medium equations and universal approximation}
Here our porous medium equation, with the Dirichlet boundary condition in a bounded smooth domain $\Omega\subset\mathbb{R}^n$ with finite boundary, read as follows

\begin{equation}\label{pme}
\begin{cases}
\rho_t = \Delta \rho^{m} + \mathbf{g}(x, t) \rho \quad \mbox{in } \Omega\times (0,T)\\
\rho_0(x,0)=\rho_0(x)\quad \mbox{in } \Omega,\\
\rho(x,t)=f\quad \mbox{in } \partial\Omega\times [0,T)
\end{cases}
\end{equation}
From the classical theory of porous medium equations, one may notice that such a system may not have any classical solution even for $\rho_0\geq 0$ and smooth (see, for example Chapter 5.3 of \cite{vazquez2007porous}). However, with additional assumptions, we have the following theorems for the existence of classical solution for \eqref{pme} (\cite{vazquez2007porous}):

\begin{theorem}\label{UAT}
Suppose both $\rho_0$ and $f$ are smooth and positive, $\mathbf{g}\in C^{\infty}$ and bounded above, then \eqref{pme} admits a unique classical solution $\rho\in C^{2}(\overline{\Omega})\cap C^1([0,T])$. Moreover, we have the following comparison principle: Suppose $0<\epsilon<\rho_0<\frac{1}{\epsilon}$, then there exists $\epsilon_0>0$ which depends on $\epsilon$, $\mathbf{g}$, $T$, $f$, such that for any $t\in [0,T]$, we have $\epsilon_0<\rho(t)<\frac{1}{\epsilon_0}$.
\end{theorem}
Hence, for such classical solution, we can use our neural network solution to approximate it, using the classical Universal Approximation Theorem (UAT) \cite{UAT}. 
\begin{theorem}
Suppose that $\rho\in  C^{2}(\overline{\Omega})\cap C^1([0,T])$ with a smooth domain $\Omega$. Let $\sigma$ be any non-polynomial function in $C^1(\mathbb{R})$, then for any $\delta>0$, we have a two-layer neural network as in \eqref{eq: PINNs}
$$
\rho^{NN}(t,x,y)=u_{\boldsymbol{\theta}}(t, x, y),
$$
such that
$$
\|\rho-\rho^{NN}\|_{C^{2}(\overline{\Omega})\cap C^1([0,T])}<\delta.
$$
\end{theorem}

\subsection{Loss function convergence}
In this section, we show that there exists a sequence of neural network solution to \eqref{pme}, such that if classical solution exists, then total loss function converges to zero. Compare with \eqref{eq: PDE loss}, \eqref{eq: IC loss},\eqref{eq: BC loss}, we define the continuous loss functions as follows:

\begin{equation}
\begin{aligned}
\tilde{\mathcal{L}}_{\text{PDE}} = \|\partial_t \rho^{NN}_j-\triangle (\rho^{NN}_j)^m-\mathbf{g}\rho^{NN}_j\|_{L^2(\Omega\times [0,T])}^2,
\end{aligned}
\label{eq: cPDE loss}
\end{equation}

\begin{equation}
\begin{aligned}
\tilde{\mathcal{L}}_{\text{IC}} =\|\rho_j^{NN}-\rho\|^2_{L^2(\Omega)},
\end{aligned}
\label{eq: cIC loss}
\end{equation}

\begin{equation}
\begin{aligned}
\tilde{\mathcal{L}}_{\text{BC}} = \|(\rho_j^{NN})^m-\rho^m\|^2_{L^2(\partial\Omega\times [0,T])},
\end{aligned}
\label{eq: cBC loss}
\end{equation}
\begin{remark}
One may notice that \eqref{eq: cBC loss} is slightly different from the discrete version boundary loss \eqref{eq: BC loss}, by a power $m$. This is due to the technicality of the analysis proof. In practice, based on the numerical experiment, the error of the final output is negligible.
\end{remark}
\begin{theorem}\label{loss}
Assume $|\mathbf{g}|\leq B$. Consider the solution $\rho\in C^{2}(\overline{\Omega})\cap C^1([0,T]) $to \eqref{pme}, with $|\rho|\leq C_0$ for some constant $C_0$ depending on $m$, $T$, $\mathbf{g}$, $\Omega$, $f$ and $\rho_0$. For given smooth non-polynomial $C^1$ activation function $\sigma$, there exists a sequence of neural network parameters $\{j\}$, such that the corresponding loss function $\tilde{\mathcal{L}}_{loss}\coloneqq \tilde{\mathcal{L}}_{PDE}+\tilde{\mathcal{L}}_{BC}+\tilde{\mathcal{L}}_{IC}\rightarrow 0$ as $j\rightarrow \infty$.
\end{theorem}

\begin{proof}
Based on Theorem \ref{UAT}, we can pick $\delta_j=1/j$, such that for the corresponding neural network functions $\rho^{NN}_{j}$, we have
$$
\|\rho-\rho^{NN}_j\|_{C^{2}(\overline{\Omega})\cap C^1([0,T])}<\delta_j.
$$
By definition, as $\rho$ is the solution to \eqref{pme} 
\begin{align*}
\begin{split}
\sqrt{\tilde{\mathcal{L}}_{PDE}}&=\|\partial_t \rho^{NN}_j-\triangle (\rho^{NN}_j)^m-\mathbf{g}\rho^{NN}_j\|_{L^2(\Omega\times [0,T])}\\
&=\|\partial_t (\rho^{NN}_j-\rho)-\triangle \left((\rho^{NN}_j)^m-\rho^m\right)-\mathbf{g}(\rho^{NN}_j-\rho)\|_{L^2(\Omega\times [0,T])}\\
&\leq \|\partial_t (\rho^{NN}_j-\rho)\|_{L^2}+\|\triangle \left((\rho^{NN}_j)^m-\rho^m\right)\|_{L^2}+\|\mathbf{g}(\rho^{NN}_j-\rho)\|_{L^2}\\
&\leq I+II+III.
\end{split}
\end{align*}
One can easily get
$$
I\leq (T\Omega)^{1/2}\|\rho-\rho^{NN}_j\|_{C^2C^1}<(T\Omega)^{1/2}\delta_j,
$$
$$
III\leq |\mathbf{g}|_{L^{\infty}}(T\Omega)^{1/2}\|\rho-\rho^{NN}_j\|_{C^2C^1}\leq B(T\Omega)^{1/2}\delta_j.
$$
For $II$, we need the following elementary inequality, which can be proved by binomial expansion and easy calculation:
\begin{equation}\label{mpower}
|a^k-b^k|\leq |a-b|\left(|a|+\max\{1,|a-b|\}\right)^k.
\end{equation}
Hence,
\begin{align*}
\begin{split}
II&\leq \|\triangle \left((\rho^{NN}_j)^m-\rho^m\right)\|_{L^2}\\
&\leq m(m-1)\|(\rho^{NN}_j)^{m-2}|\nabla\rho^{NN}_j|^2-\rho^{m-2}|\nabla\rho|^2\|_{L^2}\\
&+m\|(\rho^{NN}_j)^{m-1}\triangle \rho^{NN}_j-\rho^{m-1}\triangle \rho\|_{L^2}\\
&\leq m(m-1)\|(\rho^{NN}_j)^{m-2}|\nabla\rho^{NN}_j-\nabla\rho|^2\|_{L^2}\\
&+2m(m-1)\|(\rho^{NN}_j)^{m-2}|\nabla\rho^{NN}_j-\nabla\rho||\nabla\rho|\|_{L^2}\\
&+m(m-1)\|\left((\rho^{NN}_j)^{m-2}-\rho^{m-2}\right)|\nabla\rho|^2\|_{L^2}\\
&+m\|(\rho^{NN}_j)^{m-1}(\triangle \rho^{NN}_j-\rho)\|_{L^2}+\|\left(\rho^{NN}_j)^{m-1}-\rho^{m-1}\right)\triangle \rho\|_{L^2}
\end{split}
\end{align*}
By using \eqref{mpower}, we have
\begin{align*}
\begin{split}
II&\leq C(m,\Omega, T)\|\rho-\rho^{NN}_j\|_{C^2C^1}\|\rho_j^{NN}\|_{L^{\infty}}^{m-2}(\max\{1,\|\rho-\rho^{NN}_j\|_{C^2C^1}\}+\|\rho\|_{C^2C^1})\\
&+C(m,\Omega, T)\|\rho-\rho^{NN}_j\|_{C^2C^1}(\max\{1,\|\rho-\rho^{NN}_j\|_{C^2C^1}\}+\|\rho\|_{C^2C^1})^m.\\
\end{split}
\end{align*}
Since $\|\rho_j^{NN}\|_{L^{\infty}}\leq \|\rho\|_{L^{\infty}}+\|\rho-\rho^{NN}_j\|_{C^2C^1}\leq C_0+\delta_j$, we have
$$
II\leq C(m,\Omega, T, C_0) \delta_j.
$$
As a consequence, we have
\begin{equation}\label{LPDE}
\tilde{\mathcal{L}}_{PDE}\leq C(m,\Omega, T, C_0) \delta_j^2.
\end{equation}
In addition, we also have
\begin{equation}\label{LBC}
\tilde{\mathcal{L}}_{BC}=\|(\rho_j^{NN})^m-\rho^m\|^2_{L^2(\partial\Omega\times [0,T])}\leq (C_0+1)^m\|\rho_j^{NN}-\rho\|_{C^2C^1}|\partial\Omega|T,
\end{equation}
and
\begin{equation}\label{LIC}
\tilde{\mathcal{L}}_{IC}=\|\rho_j^{NN}-\rho\|^2_{L^2(\Omega)}\leq \|\rho_j^{NN}-\rho\|^2_{C^2C^1}|\Omega|.
\end{equation}
Combine \eqref{LPDE}, \eqref{LBC} and \eqref{LIC}, we have $\tilde{\mathcal{L}}_{loss}\rightarrow 0$ as $j\rightarrow \infty$.

\end{proof}

\subsection{PINN solution convergence}
In this section, we prove that with the parameter $\{j\}$, the neural network we found in Theorem \ref{loss} converges to the classical solution of \eqref{pme}. Actually, we have the following $L^1$ contraction theorem, which is the standard result for solutions to porous medium equations. 
\begin{theorem}
Let $\rho^{NN}$ be the neural network solution given by Theorem \ref{loss}. Assume $m$ is an odd integer, $\rho$ is the classical solution to \eqref{pme} and $|\rho|\leq C_0$, $|\mathbf{g}|_{L^{\infty}}\leq B$, then we have
\begin{equation}\label{contraction}
\|\rho^{NN}-\rho\|_{L^1(\Omega\times [0,T])}\leq  C\int_0^T \mathcal{L}(t)dt,
\end{equation}
with a constant $C$ depending on $T,m,\Omega,B$.
\end{theorem}
\begin{proof}
The proof of this based on the standard $L^1$ contraction estimate of porous medium equation, namely, if $\rho^{NN}$ satisfies the equation
\begin{equation}\label{eqNN}
    \begin{cases}
\rho^{NN}_t = \Delta (\rho^{NN})^{m} + \mathbf{g}(x, t) \rho^{NN}+er(x,t) \quad \mbox{in } \Omega\times (0,T)\\
\rho_0(x,0)=\rho^{NN}_0(x)\quad \mbox{in } \Omega,\\
\rho(x,t)=f^{NN}(x,t)\quad \mbox{in } \partial\Omega\times [0,T)
\end{cases}
\end{equation}
then we have the following:
\begin{equation}\label{L1contraction}
\begin{split}
  \int_{\Omega\times [0,T]}(\rho-\rho^{NN})^+dxdt\leq &C\int_{\Omega}(\rho_0-\rho_0^{NN})^+dx+C\int_{\partial\Omega\times [0,T]} \left(\rho^m-(\rho^{NN})^m\right)^+dxdt\\
  &+C\int_{\Omega\times [0,T]}er(x,t)^+dxdt.
  \end{split}
\end{equation}
Here $h(x,t)^+=\max\{h(x,t),0\}$, $C$ is a constant that depending on $B,T,m,\Omega$. Due to the fact that classical solution is also the entropy weak (sub- or super-)solution, one can follow step by step in \cite{kobayasi2006kinetic}, Theorem $1.1$, to get such a $L^1$ contraction. Then, to get \eqref{contraction} one can simply use the Cauchy's inequality and the fact that $\Omega$ is finite. 
\end{proof}

\subsection{Remarks regarding the theoretical assumptions}\label{subsec:rmk}
The convergence theorems in this section assume smoothness and strict positivity of the initial and boundary data, as well as boundedness of the source term $\mathbf{g}$. By classical porous medium theory \cite{vazquez2007porous},  a globally defined classical solution exists under these conditions; otherwise the equation degenerates where the density vanishes, and the problem is not classically well‑posed. However, the initial condition in this paper, like \eqref{eq: initial}, is a discontinuous characteristic function and the boundary condition is homogeneous Dirichlet, both violate the positivity requirement. Consequently, the convergence theorems do not directly apply to any of the numerical experiments, whether synthetic or real.

The analysis therefore serves a different purpose: it confirms that the PINN loss functional is internally consistent for the porous medium operator in an idealized smooth setting, and it provides a theoretical benchmark that guarantees the optimization problem is well-posed when a classical solution exists. The numerical results presented later in this paper demonstrate that the method performs well even when the theoretical assumptions are violated, but this remains an empirical observation. Extending the convergence theory to weak solutions and noisy data would provide a more rigorous analytical foundation, however, such an extension is beyond the scope of this paper, thus is left for future work.

\section{Setup of PINNs and verification on sythetic data}
\label{Preliminary}

\subsection{Recovering the proliferation rate from numerical tumor growth data}
We first try to investigate whether this PINNs framework can accurately predict parameters in the case of tumor growth models governed by numerical solutions. Specifically, we select a range of values for parameter $v$. For each value, we generate the synthetic tumor density data by numerically solving the underlying PDE model. This numerical solver is adapted from the MATLAB code developed by Xu'an Dou, which implements the front-capturing scheme proposed by Liu et al. (2018) in \cite{liu2018accurate} for tumor growth models. The generated data are treated as ground truth data and fed into the PINNs framework to train the model and recover the corresponding $v$ values. As we expect, the predicted $v$ values are consistently close to the original ones used for data generation, indicating that the model can accurately identify parameters from numerically simulated tumor growth patterns.
\subsubsection{PINNs setup}
\label{setup}
To validate the effectiveness of our PINNs framework on numerical solutions of tumor growth, we implement a custom training pipeline using PyTorch as mentioned in Algorithm~\ref{alg:pinn}. The PINNs model is designed to infer the proliferation parameter $v$ from spatiotemporal tumor density data generated by a numerical PDE solver.
The neural network consists of a fully connected multilayer perceptron (MLP) with three input nodes corresponding to spatial and temporal coordinates $(t,x,y)$, and one output node representing the predicted tumor density $u(x,y,t)$.

\textbf{Networks architecture.} We choose the neural network architecture with a depth of 3 hidden layers and each layer having a width of 64 neurons. The hyperbolic tangent function (Tanh) is chosen as our activation function. The absolute value is added to the final output to ensure biologically realistic predictions, since the tumor density should be nonnegative.

\textbf{Training settings.} The network is trained using the optimization algorithm RAdam and Xavier initialization. The learning rate is first set as $10^{-3}$, together with a StepLR scheduler that reduces the learning rate by a factor of 0.9 every 1000 epochs.

\textbf{Optimization.} Within the framework of inverse problems, the neural network architecture remains consistent with that used for forward problems. The fundamental distinction lies in the incorporation of unknown physical parameters of interest as supplementary learnable parameters within the optimization process \cite{liu2025asymptotic}. In our specific case, the proliferation rate $v$ is treated as a trainable parameter alongside the network weights and biases. We therefore aim to optimize both the network parameters $\boldsymbol{\theta}$ and the physical parameter $v$ simultaneously through the following minimization problem:

\[
(\boldsymbol{\theta^*}, v^*) = \underset{\boldsymbol{\theta}, v}{\text{argmin}} \, \mathcal{L}_{\text{total}}(\boldsymbol{\theta}, v),
\]
where $\boldsymbol{\theta^*}$ and $v^*$ refers to the optimal values of all the unknown parameters that can yield the minimized $\mathcal{L}_{\text{total}}$.

The model is trained for 60,000 epochs, and the relative error between the PINNs prediction and ground truth value is monitored to evaluate performance.
We perform this validation using several ground truth values for $v$, namely 1.7, 1.8, 1.9, 2.0, 2.1, and 2.2, each corresponding to numerically simulated tumor growth profiles. Based on hyperparameter tuning trials, the weights $w_1, w_2, w_3, w_4$ in the loss function are set as shown in Table \ref{tab:v_parameter}:

\begin{table}[htbp]
\centering

\begin{tabular}{ccccc}
\toprule
$v_{\text{true}}$ & $w_1$ & $w_2$ & $w_3$ & $w_4$ \\
\midrule
1.7 &10 & 1& 1& 50\\
1.8 &10 &1 & 1& 50\\
1.9 &10 &1 &1 & 80\\
2.0 &10 &1 &1 &100 \\
2.1 &10 &1 &1 &100 \\
2.2 &10 &1 &1 &100 \\
\bottomrule
\end{tabular}
\caption{Hyperparameter settings for different ground truth values of $v$.}
\label{tab:v_parameter}
\end{table}

\subsubsection{Results}
The results are presented in Figure~\ref{fig:prediction plot}, which confirm the reliability and robustness of our PINNs-based framework in predicting the unknown parameter $v$ from the observational data. Specifically, the predicted values of $v$ show excellent agreement with the ground truth, with a relative error around 1\% across all tested cases (Figure~\ref{fig:relative error}). By enforcing the physics constraints through the PDE residual loss and incorporating measurement data, the framework achieves a balanced fit between data consistency and physical plausibility. These findings validate the effectiveness of our approach in using PINNs to solve inverse problems in tumor growth models.

\begin{figure}[htbp]
    \centering
    \includegraphics[width=0.9\textwidth]{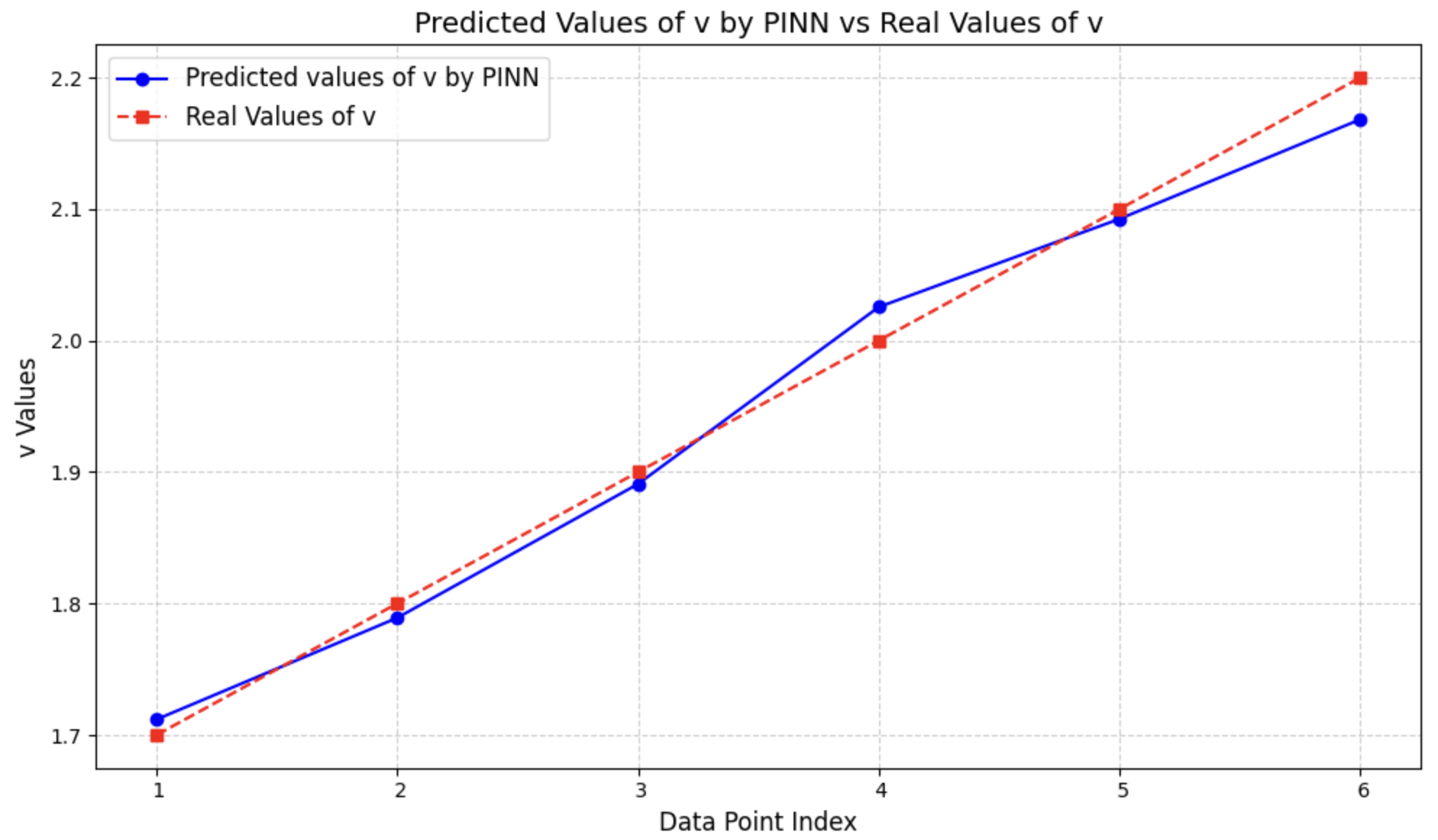}
    \caption{Predicted values of parameter $v$ by PINNs model (blue) compared with ground truth (red).}
    \label{fig:prediction plot}
\end{figure}
\begin{figure}[htbp]
    \centering
    \includegraphics[width=0.9\textwidth]{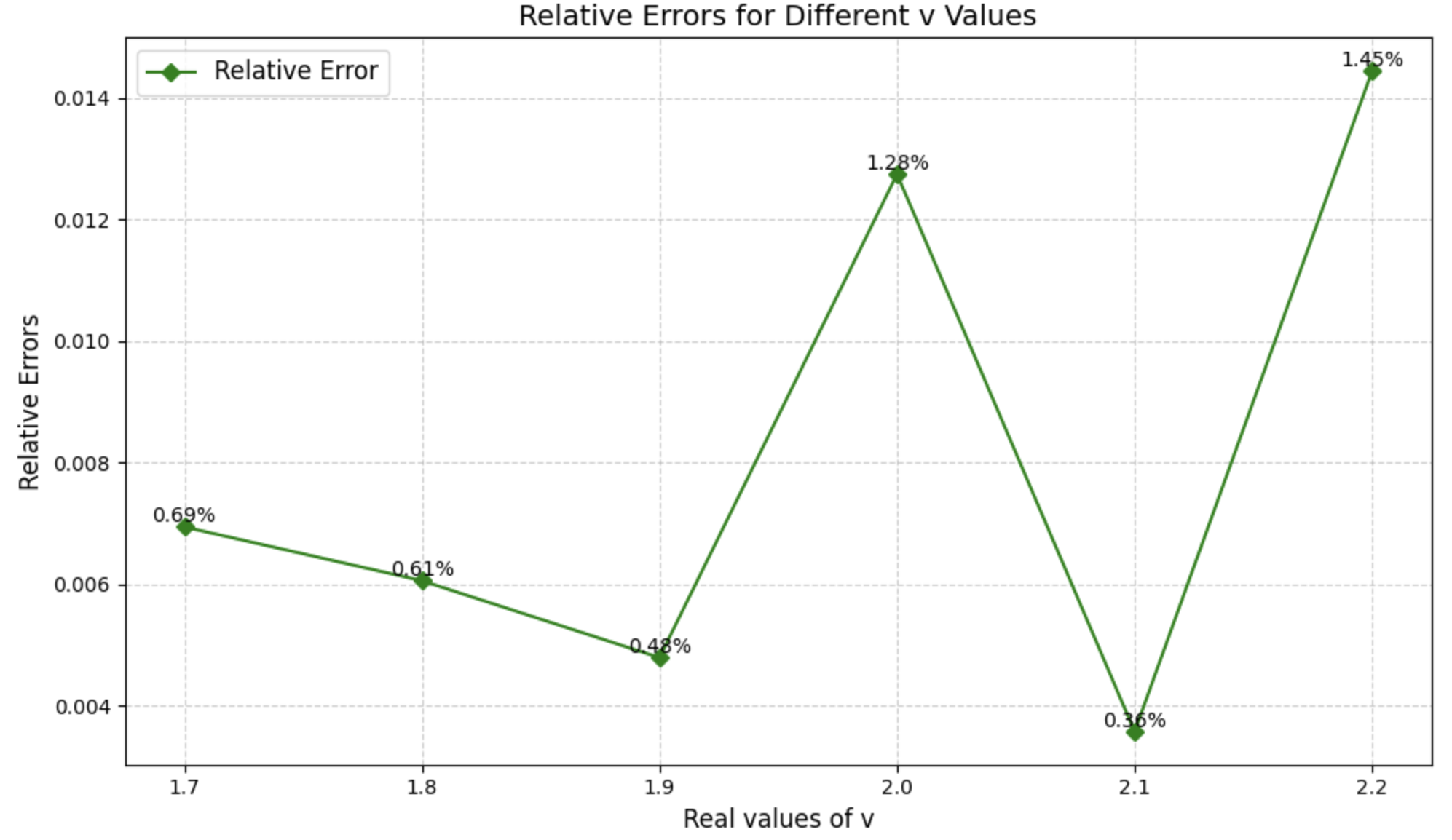}
    \caption{Relative errors for each value of $v$.}
    \label{fig:relative error}
\end{figure}

\subsection{Robustness to noisy data}
In practice, real-world data often contain noise due to experimental limitations and measurement errors. To assess the robustness of our framework, we add controlled Gaussian noise to the synthetic tumor density data:
\begin{equation}
z_{\text{noisy}} = z + \varepsilon \cdot \eta
\end{equation} 
$\varepsilon$ is the noise scaling factor, and $\eta$ follows a normal distribution:
\begin{equation}
\eta \sim \mathcal{N}(0, \sigma^2)
\end{equation}
with $\sigma$ controlling the noise standard deviation.

We examine five representative $(\varepsilon, \sigma)$ combinations to simulate different noise levels where $v_{\text{true}} = 2.1$. For each noisy dataset, the PINNs model is retrained, and we compute the relative error of the inferred $v$:
\begin{equation}
\mathcal{\varepsilon}(t) = \frac{|v_{\text{pred}} - v_{\text{true}}|}{v_{\text{true}}}
\end{equation}
and plot the error evolution in all cases.

\begin{figure}[htbp]
    \centering
    \includegraphics[width=0.9\textwidth]{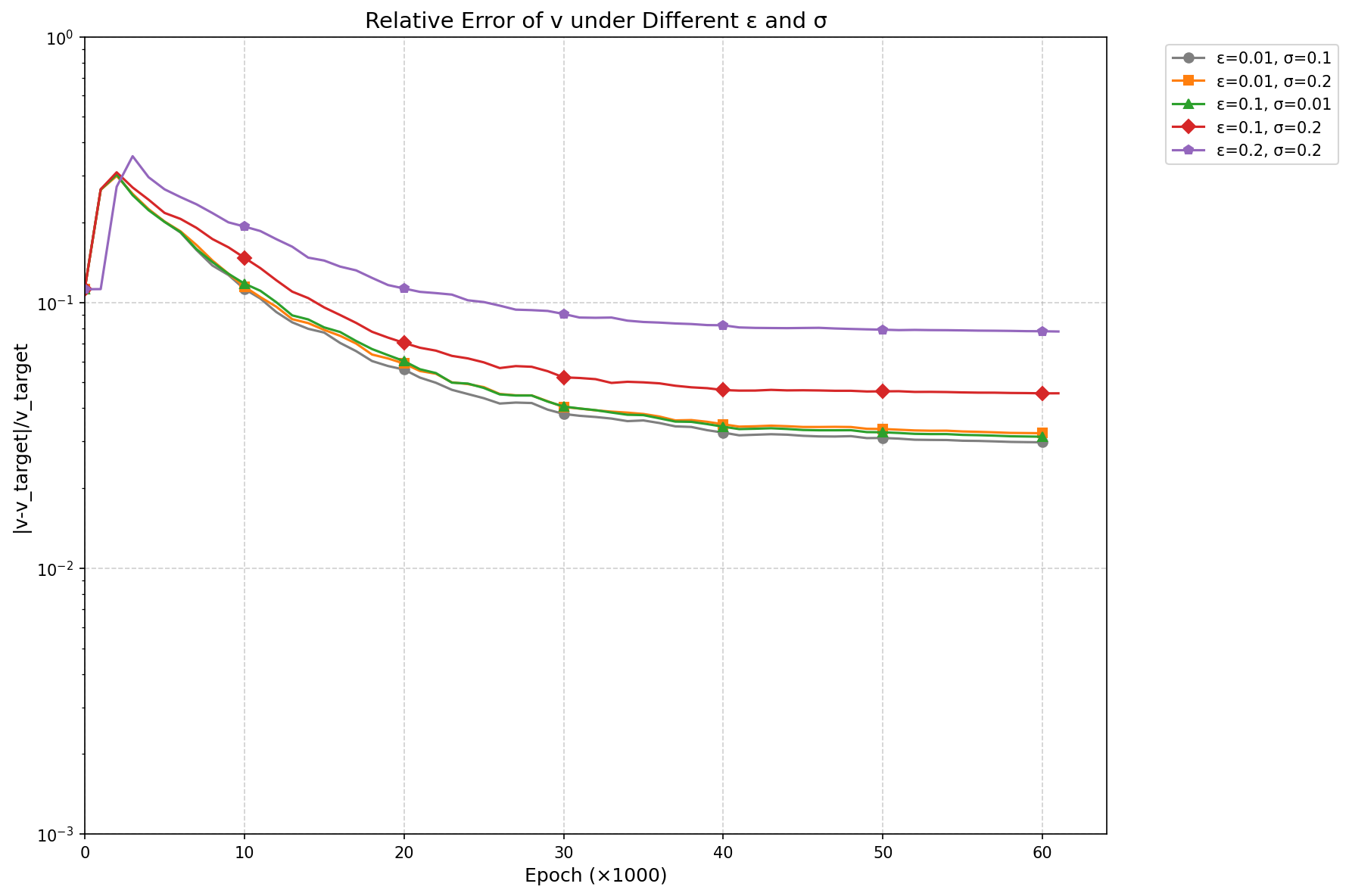}
    \caption{Relative errors between learning parameter and its target value.}
    \label{fig:error tesy}
\end{figure}

As illustrated in Figure~\ref{fig:error tesy}, the relative error decreases in all cases, showing a drop in the first 20,000 epochs followed by slower improvement. For fixed $\epsilon$, larger $\sigma$ leads to higher errors, and for fixed $\sigma$, larger $\epsilon$ also results in higher errors. Lower noise levels ($\sigma$ small) yield faster convergence and lower final errors, while high noise levels significantly slow convergence and degrade accuracy.

The prediction errors remain within acceptable bounds across all noise levels tested with less than $1\%$ error after 30000 epoch. It indicates that the real observed data may contain inherent noise and measurement uncertainties, this framework can still recover unknown parameters with reasonable accuracy. The physics-informed regularization inherent in PINNs provides a stabilizing effect that mitigates the impact of data imperfections, making the approach suitable for real-world applications where perfect, noise-free data are seldom available.

\section{Prediction of proliferation rate based on observed data}
\label{Prediction of Proliferation Rate Based on Observed Data}
\subsection{Observed data and problem setup}
\label{Observed Data and Problem Setup}
%{\color{red}lab situation-Xu}
 Having validated the capability of the proposed PINNs framework on synthetic data generated from numerical solutions, we then apply it to predict the proliferation rate $v$ based on observed tumor growth data. One can visit \cite{github1} for the details of the lab condition and dataset. The dataset records tumor growth over 18 days. Time $t$ is rescaled to the interval $[0,1]$ such that the second day corresponds to $t=0$. The spatial domain is defined as $x,y \in [-3,3]$.

The images that we can observe in lab are shown in Figure~\ref{fig:lab}, in which it is able for us to measure the radius of tumors at specific time. Assuming radial symmetry in tumor growth, we extract the tumor radius every two days based on the observed images. Due to the limitation of measurement, we cannot obtain the precise tumor density at each spatial location. Instead, we use binary labels as the real data, where regions with tumor presence are labeled as 1 and regions without are labeled as 0. Tumor necrosis at the center is neglected for simplicity. The observed data with respect to rescaled t are recorded in Table~\ref{tab:observed data}.

\begin{figure}[htbp]
    \centering
    \includegraphics[width=0.9\textwidth]{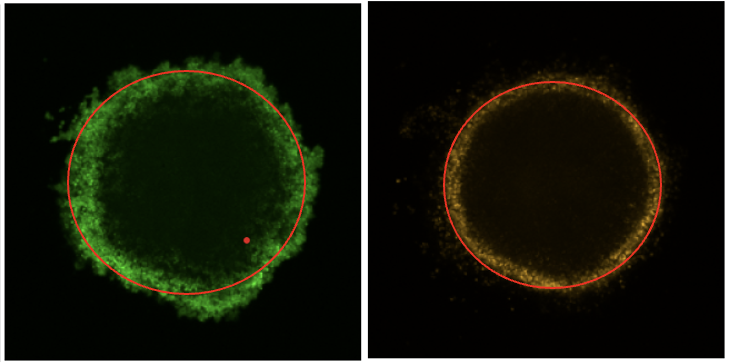}
    \caption{Observed images of tumor growth in lab, in which the tumor growth profile is marked in red circle. }
    \label{fig:lab}
\end{figure}

\begin{table}[htbp]
\centering
\begin{tabular}{|c|c|c|c|c|c|c|c|}
\hline
Time(Rescaled) & 0.25 & 0.375 & 0.5 & 0.625 & 0.75 & 0.875 & 1 \\
\hline
Tumor Radius & 0.66 & 0.97 & 1.26 & 1.48 & 1.93 & 2.13 & 2.5 \\
\hline
\end{tabular}
\caption{Real observation of tumor radius with respect to rescaled time.}
\label{tab:observed data}
\end{table}

\subsection{Implementation details}

The implementation framework is nearly identical to the algorithm described in the validation stage (see Algorithm~\ref{alg:pinn}). The similar PINNs architecture is adopted as mentioned in \ref{setup}.

The main difference lies in the formulation of the data loss. Since the observed data are binary-valued (0/1), we replace the mean-squared error (MSE) loss with the binary cross-entropy (BCE) loss:

\begin{equation}
\begin{aligned}
\mathcal{L}_{\text{data}} = \frac{1}{N_2} \sum_{i}\left[
    -y_i \log(\hat{y}_i) - (1 - y_i)\log(1 - \hat{y}_i)
\right],
\end{aligned}
\label{eq: BCE loss}
\end{equation}

where $N_2$ is the number of observed data points ($N_2 = 200$ in our study), $y_i \in \{0,1\}$ is the observed binary label and $\hat{y}_i$ is the output of neural network at the same location and time. This choice of loss function better reflects the nature of the data and enables the model to accurately identify tumor regions without requiring precise density values. Specifically, based on the hyperparameter tuning trials, the weights $w_1,w_2,w_3,w_4$ were assigned as 1, 1, 1, and 5, respectively, to emphasize more about the importance of data loss. 

Moreover, we use the first five time points (i.e., $t=0, 0.125, 0.25, 0.375, 0.5$) as training data to let PINNs learn the proliferation rate $v$, while the last two time points ($t=0.875$ and $t=1$) are used as testing data to evaluate the prediction capability of the trained model. The aim is to examine whether the inferred $v$ can accurately predict tumor growth in the future. To assess the accuracy of prediction, we compare the predicted tumor radius at $t=0.875$ and $t=1$ by PINNs with the ground truth values(i.e., $2.13, 2.5$) and compute the \emph{relative errors} between them.

\subsection{Results}
Under the framework of PINNs, the predicted value of proliferation rate $v$ with respect to interations is shown in Figure~\ref{fig:prediction_v}, suggesting the convergent value of 3.1264.

\begin{figure}[htbp]
    \centering
    \includegraphics[width=0.9\textwidth]{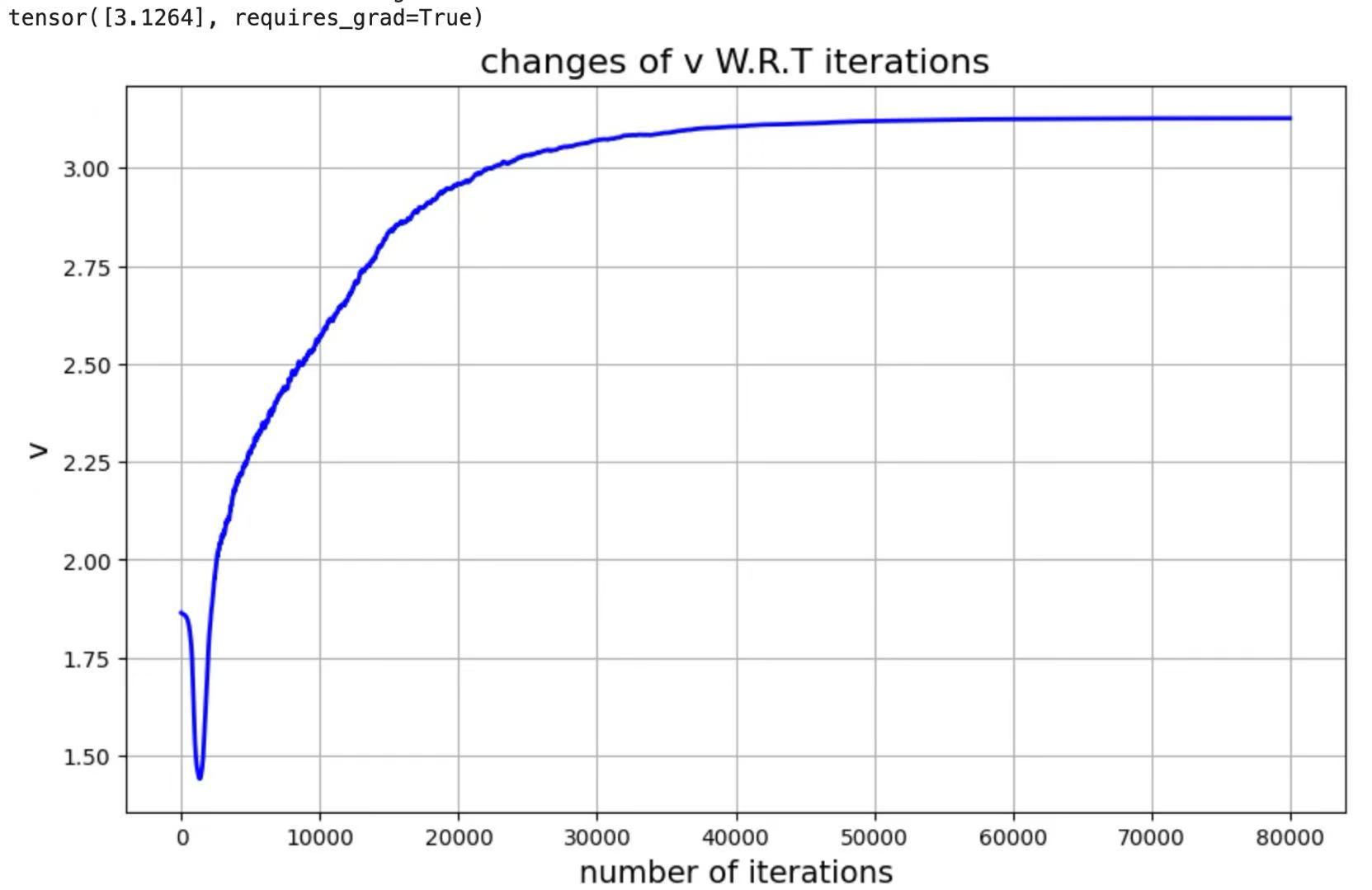}
    \caption{The evolution of $v$ across successive iterations, showing a convergent trend to 3.1264 over 0 to 80000 iterations. }
    \label{fig:prediction_v}
\end{figure}

\begin{figure}[htbp]
    \centering
    \includegraphics[width=0.9\textwidth]{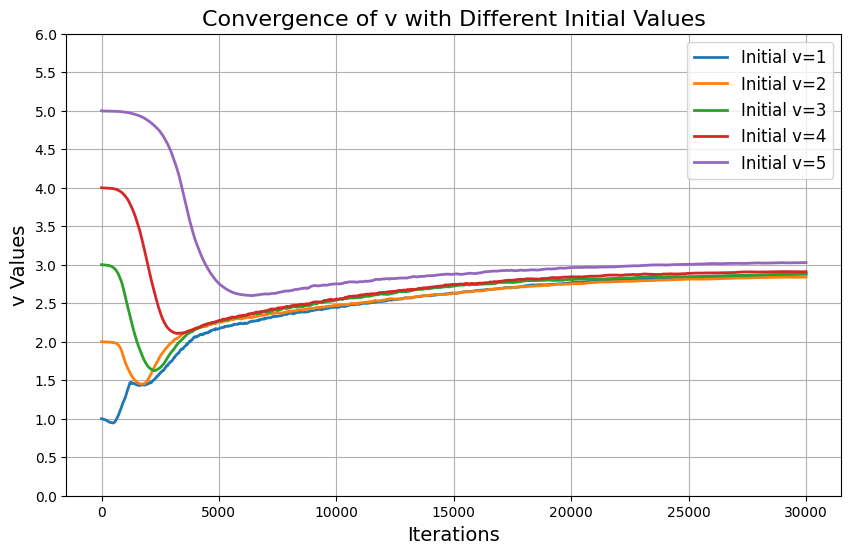}
    \caption{The convergence of $v$ with different initial values across  iterations}
    \label{fig:duo_v}
\end{figure}

To evaluate the robustness of the parameter estimation to initialization, we conduct experiments with varying initial values for parameter 
$v ={1,2,3,4,5}$. The model is trained on the experimental dataset described in Section \ref{Observed Data and Problem Setup}, using the binary cross-entropy (BCE) loss. As shown in (Figure~\ref{fig:duo_v}), the estimated values of $v$ converge to a narrow range around approximately 2.9 - 3.0 after ~25,000 epochs, indicating that the final solution exhibits a degree of invariance to the chosen starting point under the specified experimental conditions.

To evaluate the prediction performance of the trained PINNs model on real data, we use the inferred value of the proliferation rate $v$ to simulate tumor growth forward in time using the previous developed numerical solver by Xu'an Dou. Based on the trained model, we predict the tumor profile and tumor boundary at time points $t=0.875$ and $t=1$, which are depicted in Figure~\ref{fig:0.875} and Figure~\ref{fig:1}. 

\begin{figure}[htbp]
    \centering
    \includegraphics[width=1\textwidth]{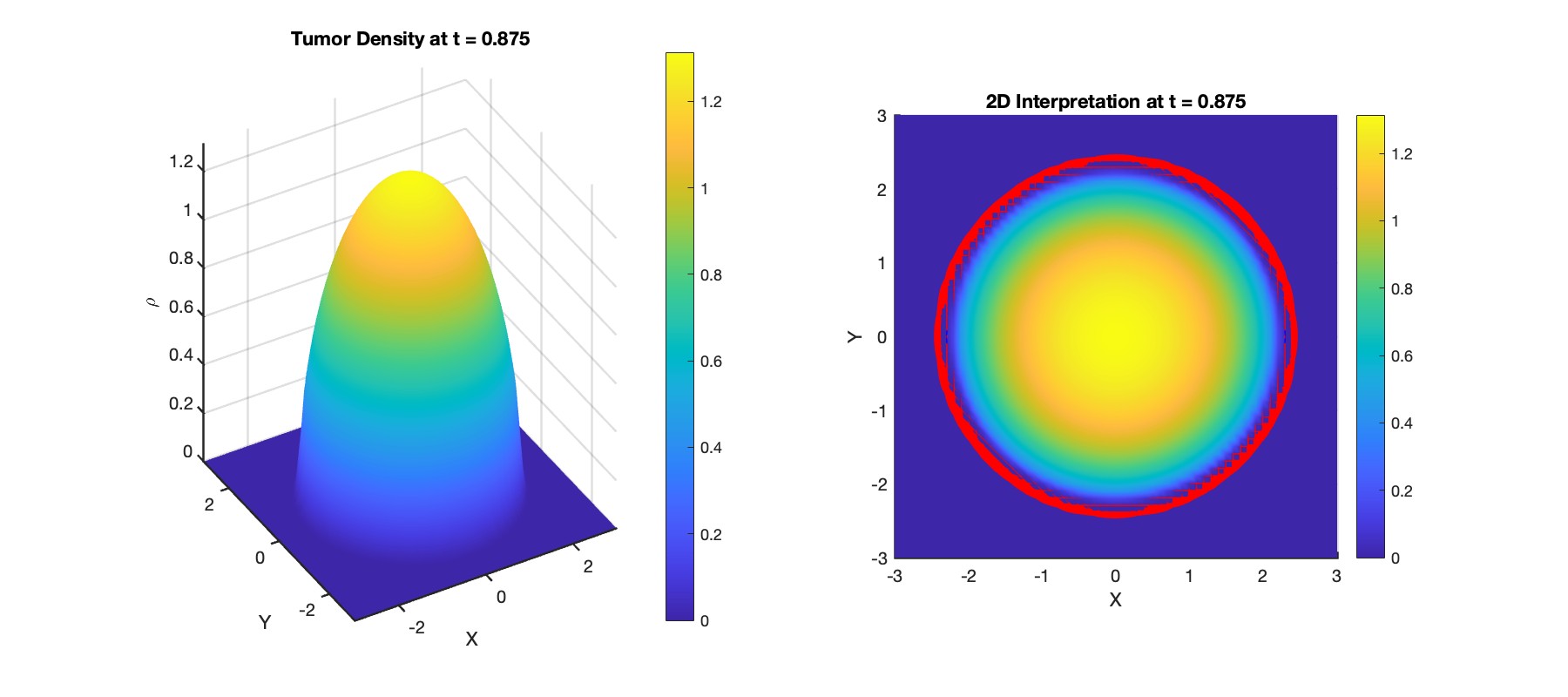}
    \caption{The left panel displays the spatial distribution of tumor density at time t = 0.875, where color intensity or numerical values represent varying density levels. The right panel provides a simplified 2D-interpretation: the red circle indicates the threshold-radius of predicted tumor presence versus absence at t=0.875.}
    \label{fig:0.875}
\end{figure}

\begin{figure}[htbp]
    \centering
    \includegraphics[width=1\textwidth]{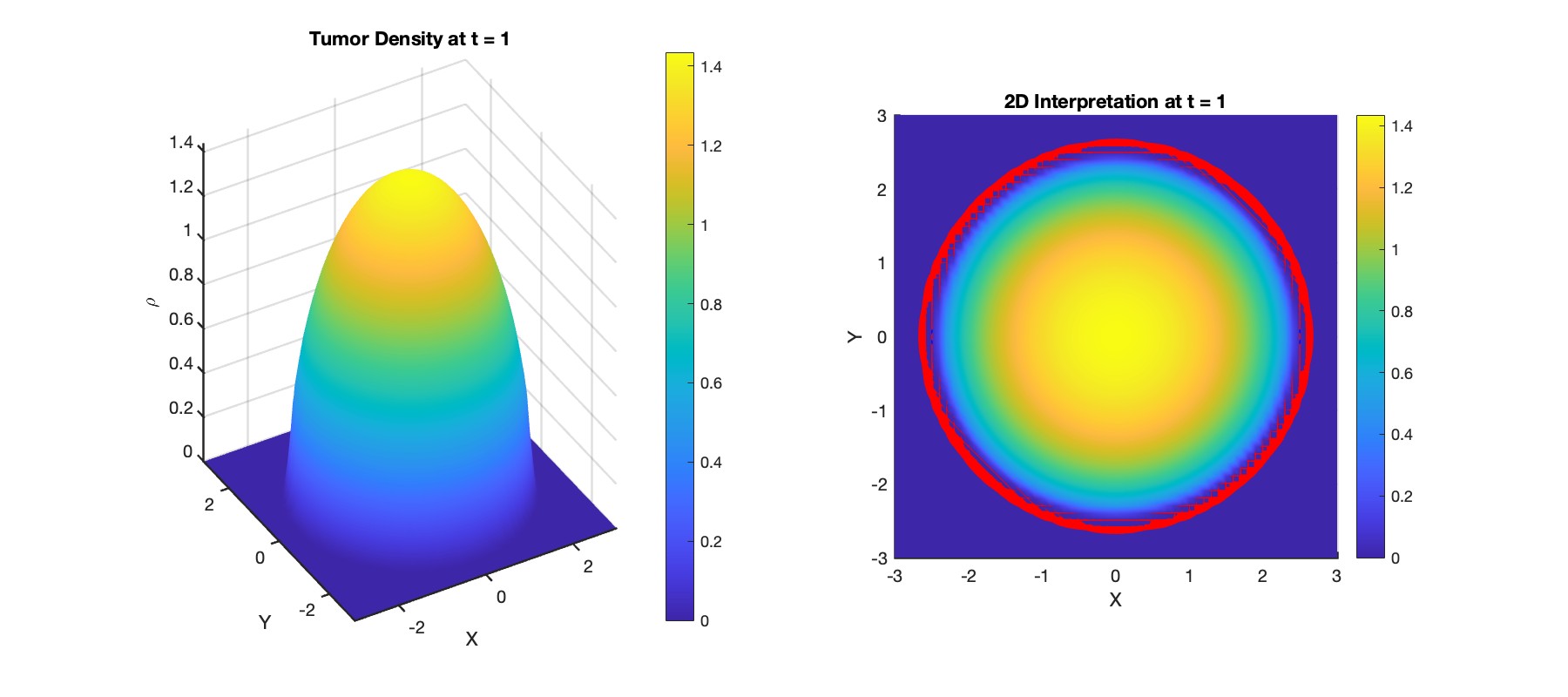}
    \caption{The left panel displays the spatial distribution of tumor density at time t = 1, where color intensity or numerical values represent varying density levels. The right panel provides a simplified 2D-interpretation: the red circle indicates the threshold-radius of predicted tumor presence versus absence at t=1. }
    \label{fig:1}
\end{figure}

Visually, the predicted tumor shapes align well with the observed data. Moreover, we plot the red circle with tumor density equal to 0.1 in the 2D-interpretation to denote the threshold of tumor presence versus tumor absence. Further more, to quantitatively assess the accuracy, we compute the relative error between the predicted and observed tumor radius at $t=0.875$ and $t=1$. The predicted tumor radius and relative error compared to the observed data (Table~\ref{tab:observed data}) are recorded in Table~\ref{tab:radius_comparison}. 
\begin{table}[h]
\centering
\begin{tabular}{l S[table-format=1.4] S[table-format=1.4] S[table-format=2.2]}
\toprule
{Time Point} & {Predicted Radius} & {Observed Radius} & {Relative Error (\%)} \\
\midrule
$t = 0.875$ & 2.2308 & 2.13 & 4.732 \\
$t = 1.0$   & 2.4426 & 2.5 & 2.296 \\
\bottomrule
\end{tabular}
\caption{Comparison between predicted and observed tumor radius with relative error.}
\label{tab:radius_comparison}
\end{table}

According to Table~\ref{tab:radius_comparison}, The relative errors are found to be small, which are 4.732\% for $t = 0.875$ and 2.296\% for $t = 1$, indicating that the trained model has good predictive capability and is able to capture the underlying dynamics of tumor progression even under limited data. This demonstrates the reliability and generalization ability of our PINNs framework when applied to real-world tumor growth data.

\subsection{Physics-Informed DeepONet for parameter identification}

In addition to the standard PINNs framework, we also explore an alternative deep learning architecture known as the Physics-Informed Deep Operator Network (PI-DeepONet) for the same inverse problem of estimating the proliferation rate \(v\). DeepONet, introduced by Lu et al.~\cite{lu2021learning}, is founded on the universal approximation theorem for operators, which guarantees that a neural network can approximate nonlinear operators mapping between infinite-dimensional function spaces. Unlike conventional PINNs that learn a single solution function, DeepONet learns the solution \emph{operator}: given the initial condition as input, it can predict the entire spatiotemporal tumor density field. This operator learning paradigm could offer advantages when solutions are required for multiple initial conditions or when real-time prediction is needed \cite{he2023novel}.

A standard DeepONet consists of two subnetworks: a branch net that encodes the input function (the initial tumor density \(\rho_0(x,y)\) evaluated at sensor points), and a trunk net that encodes the query coordinates \((x,y,t)\). Their dot product with a bias yields the predicted solution. To adapt DeepONet for inverse problems, we treat the unknown proliferation rate \(v\) as a learnable scalar parameter embedded within the PDE residual loss, analogous to the PINNs approach. The total loss function follows the same structure as defined in equation~\ref{eq: loss funtion}, comprising the PDE residual loss, initial condition loss, boundary condition loss, and data loss. The key distinction lies in the network architecture: the branch net encodes the initial condition vector of length \(n_{\text{sensors}} = N_x \times N_y\) (\(61 \times 61 = 3721\) sensor points), while the trunk net takes the three-dimensional coordinate input \((t,x,y)\). Both subnets employ 3 hidden layers with 128 neurons each and Tanh activation, producing a shared output dimension \(p = 64\). The network is implemented in PyTorch and trained using the RAdam optimizer with a StepLR scheduler (gamma = 0.5 every 10,000 epochs) for a total of 50,000 epochs. The loss weights are set to \(w_1 = 5\), \(w_2 = 1\), \(w_3 = 1\), and \(w_4 = 1\). The initial guess for \(v\) is set to 0. 

\begin{figure}[htbp]
    \centering
    % Replace with your actual figure file
    \includegraphics[width=1\textwidth]{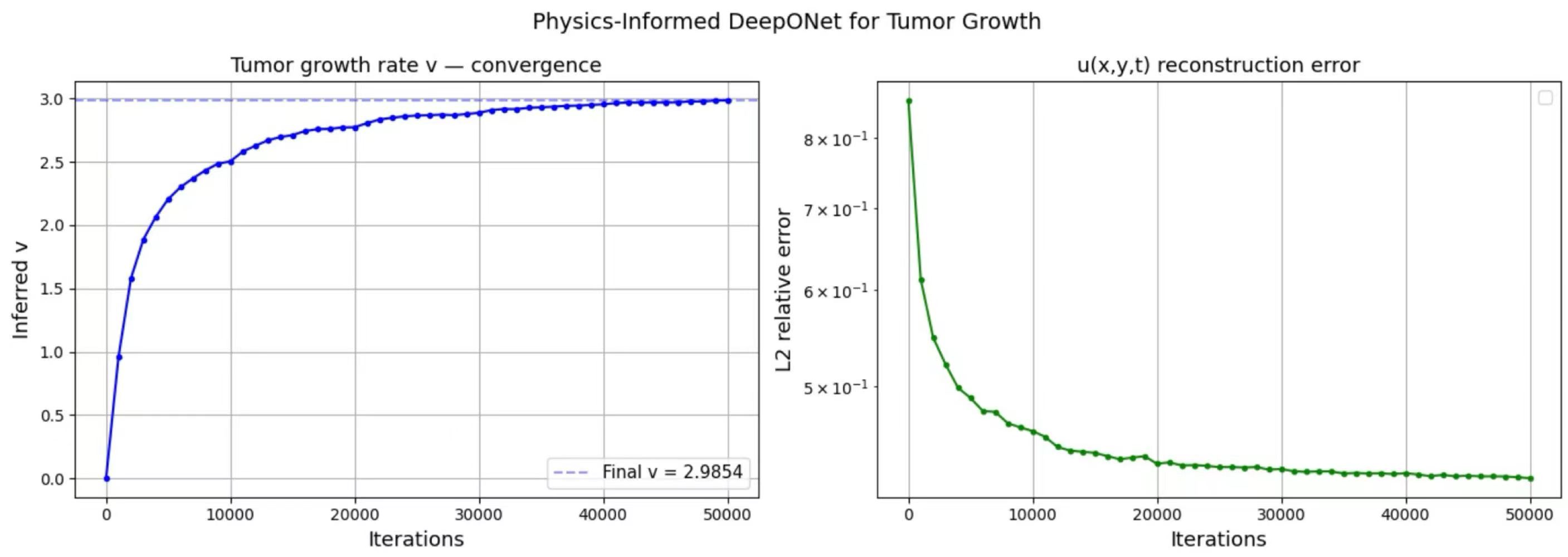}
    \caption{Physics-Informed DeepONet results for parameter identification. (a) Convergence of the inferred proliferation rate \(v\). (b) L2 relative error of the predicted tumor density field \(\rho(x,y,t)\) over training iterations.}
    \label{fig:deeponet_results}
\end{figure}

Figure~\ref{fig:deeponet_results} presents the convergence behavior of the inferred proliferation rate \(v\) using PI-DeepONet, alongside the reconstruction error of the tumor density field. The inferred value of \(v\) converges to a value nearly identical to the PINNs result. The L2 relative error in reconstructing the full spatiotemporal tumor density field decreases steadily over training, achieving a final accuracy comparable to that of the standard PINN framework. This indicates that PI-DeepONet achieves comparable accuracy to the standard PINNs framework for this inverse tumor growth problem. The operator learning perspective of DeepONet may offer additional flexibility when dealing with varying initial conditions or when rapid inference across multiple scenarios is required, while maintaining the same physics-informed regularization that ensures robust parameter estimation from sparse and noisy data.

\section{Application to multiple unknown parameters}
\label{Application to Multiple Unknown Parameters}
\subsection{Extension to spatially varying proliferation rate}
To further investigate the applicability of our PINNs framework, we extended the original model by allowing the proliferation rate to vary spatially. Biologically, this assumption is reasonable, as the proliferation of tumor cells may be influenced by local nutrient concentration, which can vary across space. Under the radial symmetry assumption, we introduce a spatially dependent proliferation rate modeled as

\begin{equation}
\begin{aligned}
g(x, y) = v_1 + v_2 \sin(\sqrt{x^2 + y^2}),
\end{aligned}
\label{eq: rate_xy}
\end{equation}

where $v_1$ and $v_2$ are two unknown parameters to be inferred.

Then the governing equation becomes:
\begin{equation}
\begin{aligned}
\rho_t - \Delta(\rho^3) = \left(v_1 + v_2 \sin(\sqrt{x^2 + y^2})\right) \rho,
\end{aligned}
\label{eq: duo_pde}
\end{equation}

subject to the same initial condition in Equation \eqref{eq: initial} and homogeneous Dirichlet boundary conditions as described previously.

We use almost the same training strategy and algorithmic framework (see Algorithm~\ref{alg:pinn}) to learn the parameters $v_1$ and $v_2$ from the early-time binary tumor data (training data up to $t = 0.75$), in which the two parameters are first assigned initial values and then updated together with the parameters in neural network. The data loss function is again based on the binary cross-entropy (BCE) loss, which is well-suited for 0/1-type observed data. Moreover, the weights are assigned as 1,1,1 and 4 respectively according to hyperparameter tuning trials. The results are plotted in Figure~\ref{fig:duo_v1_v2}, showing the convergent trend for both $v_1$ and $v_2$ to 7.0968 and -5.9086 respectively.

\begin{figure}[htbp]
    \centering
    \includegraphics[width=1\textwidth]{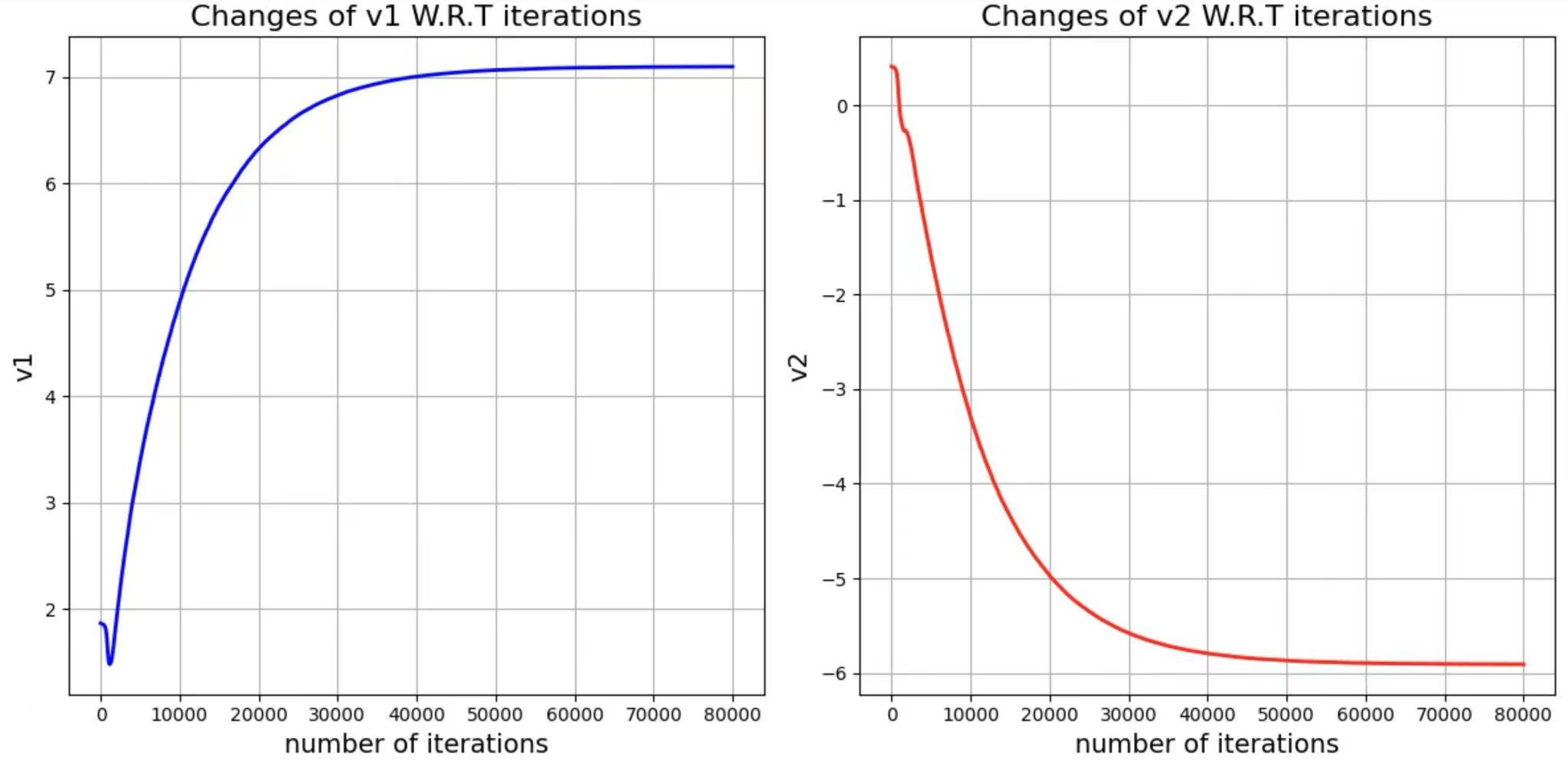}
    \caption{The evolution of the variables $v_1$ and $v_2$ across successive iterations, showing a convergent trend to 7.0968 and -5.9086 over 0 to 80000 iterations.}
    \label{fig:duo_v1_v2}
\end{figure}

After the training, we simulate tumor growth forward using the learned values of $v_1$ and $v_2$ and compare the predicted tumor boundaries at $t=0.875$ and $t=1$ against the ground truth values, which are recorded in Table~\ref{tab:radius_comparison_duo}.

\begin{table}[h]
\centering
\begin{tabular}{l S[table-format=1.4] S[table-format=1.4] S[table-format=2.2]}
\toprule
{Time Point} & {Predicted Radius} & {Observed Radius} & {Relative Error (\%)} \\
\midrule
$t = 0.875$ & 2.2149 & 2.13 & 3.986 \\
$t = 1.0$   & 2.3956 & 2.5 & 4.176 \\
\bottomrule
\end{tabular}
\caption{Comparison between predicted and observed tumor radius with relative error.}
\label{tab:radius_comparison_duo}
\end{table}

As shown in Figure~\ref{fig:duo_0.875} and Figure~\ref{fig:duo_1}, the predicted tumor radius closely matches the real ones, and the computed relative errors are also small, namely 3.986\% for $t=0.875$ and 4.176\% for $t=1$. This demonstrates that the proposed PINNs framework remains robust and reliable even when the underlying model becomes more complex with spatially varying parameters.

\begin{figure}[htbp]
    \centering
    \includegraphics[width=1\textwidth]{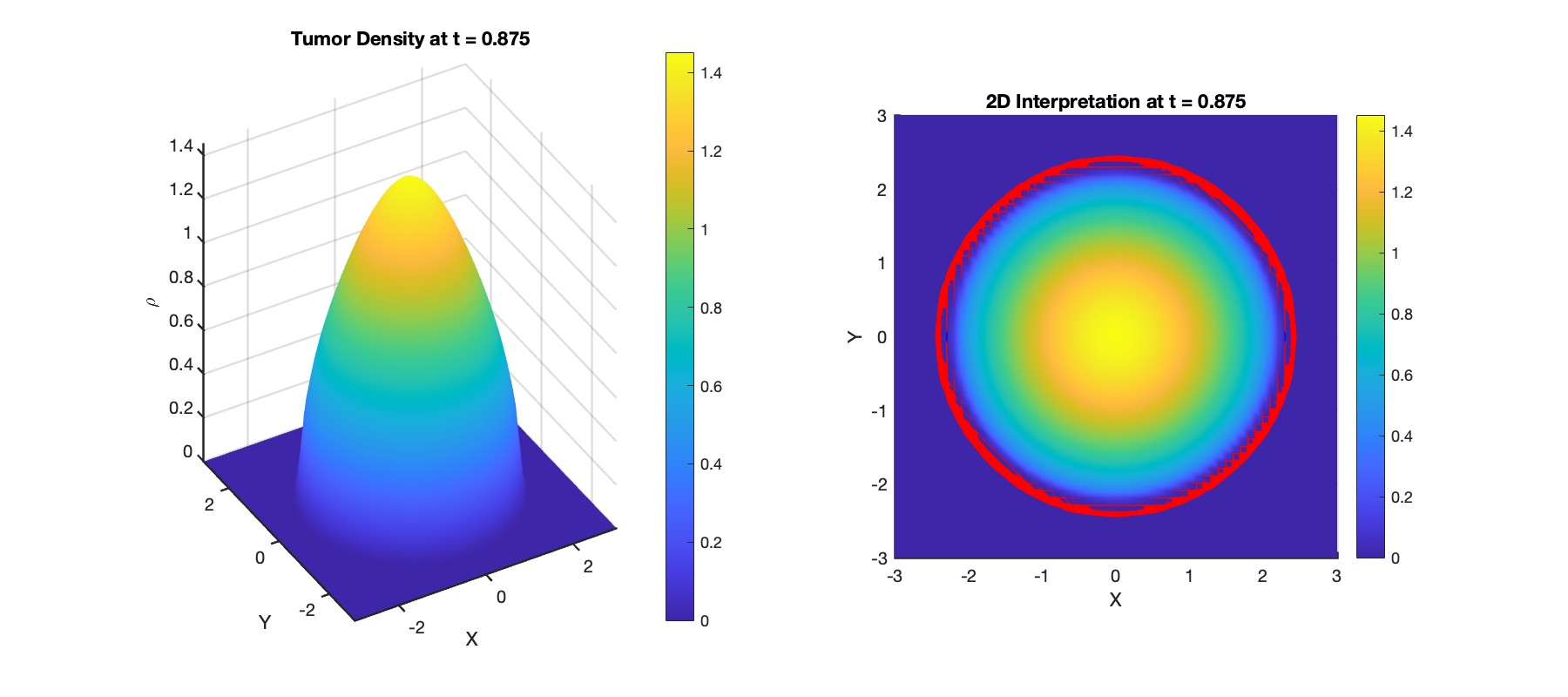}
    \caption{The left panel displays the spatial distribution of tumor density at time t = 0.875, where color intensity or numerical values represent varying density levels. The right panel provides a simplified 2D interpretation: the red circle indicates the predicted tumor presence versus absence, the threshold-radius at t=0.875. }
    \label{fig:duo_0.875}
\end{figure}

\begin{figure}[htbp]
    \centering
    \includegraphics[width=1\textwidth]{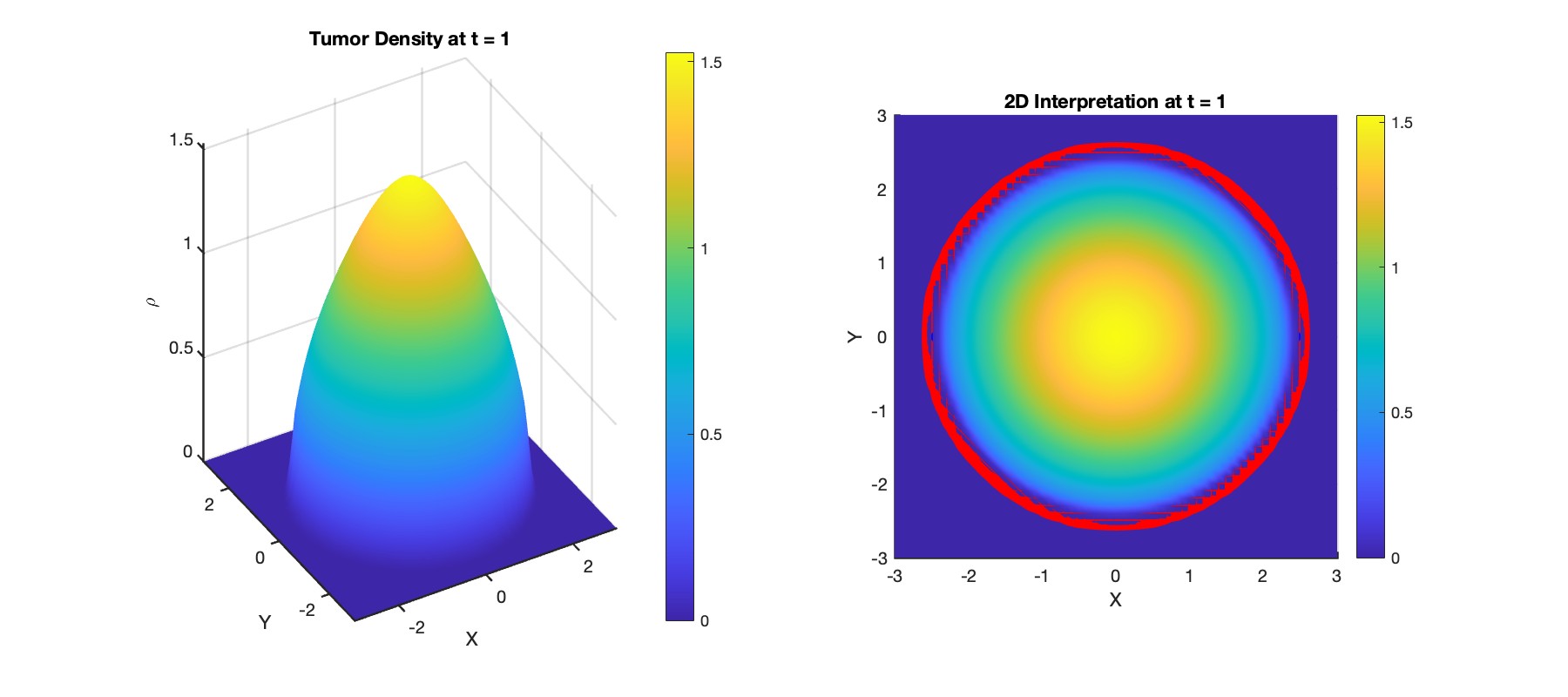}
    \caption{The left panel displays the spatial distribution of tumor density at time t = 1, where color intensity or numerical values represent varying density levels. The right panel provides a simplified 2D interpretation: the red circle indicates the predicted tumor presence versus absence, the threshold-radius at t=1. }
    \label{fig:duo_1}
\end{figure}

\subsection{Incorporating an unknown initial density parameter}

In this subsection, we investigate the case where the initial tumor density is unknown. Specifically, we modify the initial condition in Equation~\eqref{eq: initial} by introducing a parameter $a$ representing the initial density within the tumor region:

\begin{equation}
\rho_0(x,y) = 
\begin{cases} 
a, & \text{if } x^2 + y^2 < 0.25, \\
0, & \text{otherwise},
\end{cases}
\end{equation}

while keeping the original PDE (Equation~\eqref{eq: 3_PDE}) with a single unknown parameter $v$. This setup reflects scenarios where the initial tumor density is not directly measurable but must be inferred alongside the proliferation rate.

To train the PINNs model, we adopt the same framework as in Algorithm~\ref{alg:pinn}, together with the application of binary cross-entropy (BCE) loss, but now optimize both $v$ and $a$ simultaneously. The weights for the loss components are also tuned to 1, 1, 1 and 5 to balance the contributions from the physical constraints and the data fidelity.

The results demonstrate that the model successfully recovers both $v$ and $a$. Figure \ref{fig:va} shows the convergence trajectories of the parameters during training, with the convergence of $v$ to 3.1441 and $a$ to 0.3754.

\begin{figure}[h]
\centering
\includegraphics[width=1\textwidth]{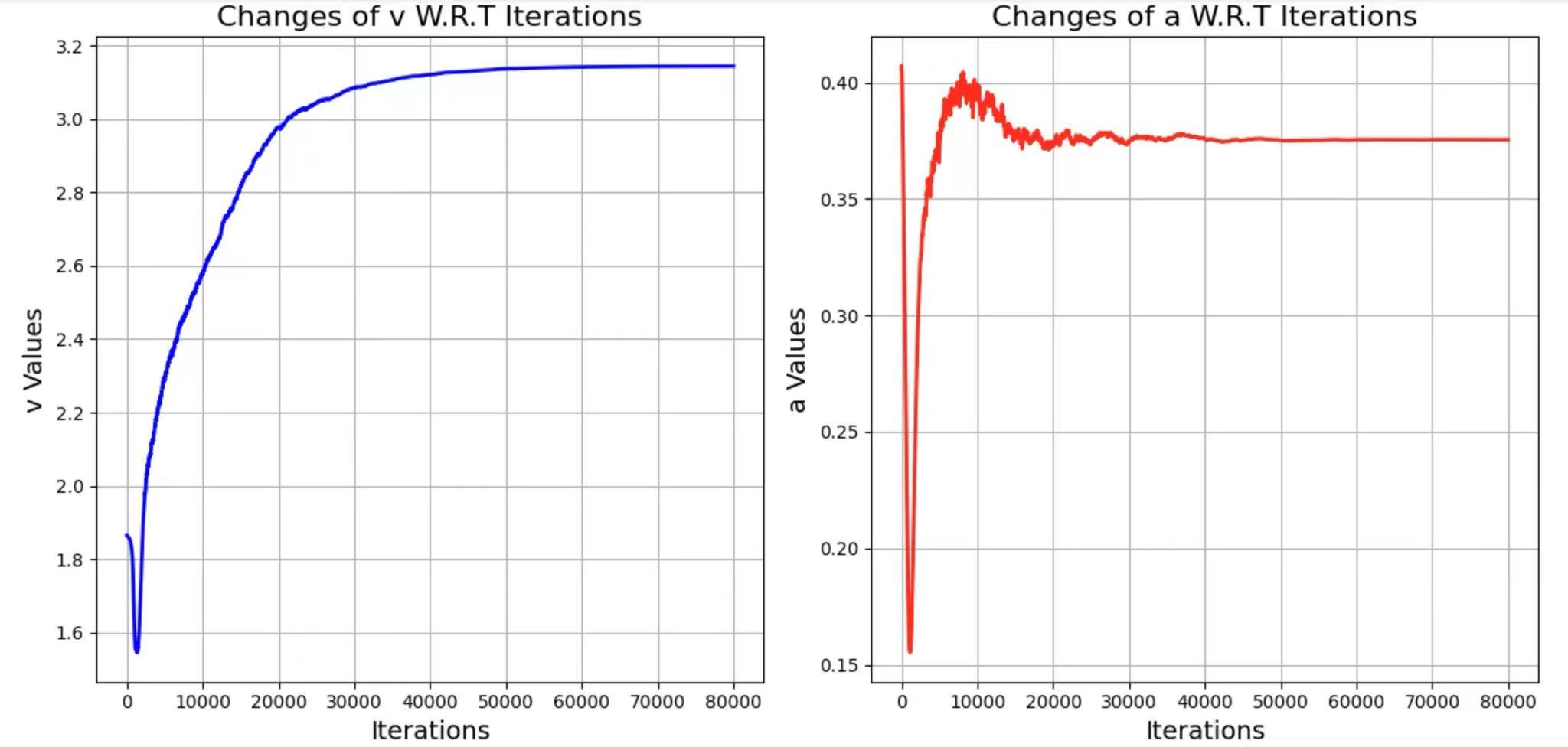}
\caption{The evolution of the variables $v$ and $a$ across successive iterations, showing a convergent
trend to 3.1441 and 0.3754 over 0 to 80000 iterations.}
\label{fig:va}
\end{figure}

This extension highlights the flexibility of the PINNs framework in handling additional unknown parameters, even when they are embedded in the initial conditions rather than the governing PDE.

\section{Conclusion}
\label{Conclusion}
This study demonstrates the potential of adopting PINNs for parameters estimation in tumor growth models. The framework has been validated through numerical experiments, showing an accurate recovery of proliferation rates from synthetic data. Furthermore, when applied to experimental tumor data measured in lab, the method 
can yield tumor radius predictions within 
a relatively low error compared with observed data. These results suggest that PINNs offer a powerful approach especially under the challenging scenarios when only scarce and noisy data is available in biomedical applications, providing a useful tool to solve both forward and inverse problems for tumor growth models. This research has enlightened a new pathway to tumor growth study and modeling with the use of real-life lab measurement data. 

While this work validates the PINNs and DeepONet frameworks for parameter identification in tumor growth models, it is essential to acknowledge their operational limitations. First, regarding noise and data sparsity, while our findings indicate that the physics-informed regularization effectively stabilizes the training process, the performance may depend on the signal-to-noise ratio. Beyond a certain noise threshold, the PDE residual loss may conflict with the data loss, potentially leading to slower convergence. Second, concerning scalability, our current implementation focuses on a 2D spatial domain. Extending it to 3D and more complex models increases the computational demand for sampling and neural network training. Future research will explore higher dimensional problems. Finally, the performance may depend on the initialization of the neural network parameters; while we observed consistent convergence across various initial guesses in our experiments, more complex landscapes require advanced optimization strategies, such as multi-stage training or adaptive weight tuning. These topics will be studied in our future work.

\section*{Acknowledgement}
This paper is supported by National Key R\&D Program of China (2021YFA1001200). All the codes and data can be found at {\cite{github1}}. We gratefully acknowledge Xu'an Dou for generously providing the numerical solver codes that underpins this research, which can also be found at \cite{github1}. We would also like to thank Yingxue Li's lab for providing the lab data for tumor. L.~Liu acknowledges the support by National Key R\&D Program of China (2021YFA1001200), Ministry of Science and Technology in China, General Research Fund (14301423 \& 14307125) funded by Research Grants Council of Hong Kong. X.~Xu was partially supported by National Key R\&D Program of China (2021YFA1001200) and Kunshan Shuangchuang Talent Program (kssc202102066).

\appendix
\section{Supplementary Hyperparameter Analysis}
\label{app:robustness}

The loss weights \(w_1, w_2, w_3, w_4\) in equation~\ref{eq: loss funtion} balance the contributions of the PDE residual, initial condition, boundary condition, and data fidelity. In principle, optimal weights should reflect the relative scales of each loss component and the desired trade-off between physical constraints and data fit. We initially performed a grid search over a small range of candidate weights, monitoring both the final L2 reconstruction error and the convergence of the inferred proliferation rate \(v\). The chosen configuration in our work was selected specifically to better minimize the L2 relative error and the total loss function value, and it additionally exhibited very stable convergence behavior.

To further demonstrate that our results are not overly sensitive to the exact weight values, we conducted additional experiments with multiple alternative weight combinations. The detailed configurations and resulting performance are summarized in Table~\ref{tab:weight_sensitivity} and Table~\ref{tab:weight_sensitivity2}. Table~\ref{tab:weight_sensitivity} reports results on synthetic data. In this controlled experiment, the true underlying parameter is set to \( v = 1.7 \). By varying the loss weight configurations, we examine how the predicted value \( v_{\text{predict}} \) deviates from the ground truth. Table~\ref{tab:weight_sensitivity2} reports results on real experimental data. Also, we examine the variation in \( v_{\text{predict}} \) across different weight configurations. The purpose is to verify that our inference method is not overly sensitive to the specific choice of weights.

Both tables share the same column structure:

\begin{itemize}
    \item \( l_i \) (PDE Loss Weight): Weight assigned to the physics-informed loss term, which enforces the governing partial differential equation. Higher values force the solution to adhere more strictly to the PDE.

    \item \( l_{\text{initial}} \) (Initial Condition Loss Weight): Weight for the loss term that enforces the initial condition at \( t = 0 \). This ensures the solution matches the prescribed initial state.

    \item \( l_{\text{dn}} \) (Downstream Boundary Loss Weight): Weight for the boundary condition loss on the \textit{downstream} boundary.

    \item \( l_u \) (Upstream Boundary Loss Weight): Weight for the boundary condition loss on the \textit{upstream} boundary.

    \item \( l_l \) (Left Boundary Loss Weight): Weight for the boundary condition loss on the left side of the domain.

    \item \( l_r \) (Right Boundary Loss Weight): Weight for the boundary condition loss on the right side of the domain.

    \item \( l_{\text{data}} \) (Data Loss Weight): Weight assigned to the data loss term, which measures the discrepancy between model predictions and actual measurements. Increasing this weight forces the solution to fit the observed data more closely.

\end{itemize}

While the final inferred \(v\) varies slightly, all tested combinations yield results within a narrow range, confirming the robustness of our conclusions to reasonable variations in loss weights.

\begin{table}[htbp]
    \centering
    \caption{Different loss weight configurations for the synthetic data (v=1.7).}
    \label{tab:weight_sensitivity}
    \begin{tabular}{cccccccc}
        \toprule
        \(l_i\) & \(l_{\text{initial}}\) & \(l_{\text{dn}}\) & \(l_{\text{u}}\) & \(l_{\text{l}}\) & \(l_{\text{r}}\) & \(l_{\text{data}}\) & \(v_{\text{predict}}\) \\
        \midrule
        10 & 1.0 & 1.0 & 1.0 & 1.0 & 1.0 & 10 & 1.4941 \\
        10 & 1.0 & 1.0 & 1.0 & 1.0 & 1.0 & 50 & 1.7118 \\
        10 & 1.0 & 1.0 & 1.0 & 1.0 & 1.0 & 100 & 1.7650 \\
        40 & 1.0 & 1.0 & 1.0 & 1.0 & 1.0 & 40 & 1.5357 \\
        \bottomrule
    \end{tabular}
\end{table}

\begin{table}[htbp]
    \centering
    \caption{Different loss weight configurations for the real data.}
    \label{tab:weight_sensitivity2}
    \begin{tabular}{cccccccc}
        \toprule
        \(l_i\) & \(l_{\text{initial}}\) & \(l_{\text{dn}}\) & \(l_{\text{u}}\) & \(l_{\text{l}}\) & \(l_{\text{r}}\) & \(l_{\text{data}}\) & \(v_{\text{predict}}\) \\
        \midrule
        1.0 & 1.0 & 1.0 & 1.0 & 1.0 & 1.0 & 1.0 & 2.3872 \\
        1.0 & 1.0 & 1.0 & 1.0 & 1.0 & 1.0 & 5 & 3.1264 \\
        1.0 & 1.0 & 1.0 & 1.0 & 1.0 & 1.0 & 10 &3.2620 \\
        5 & 1.0 & 1.0 & 1.0 & 1.0 & 1.0 & 10 & 2.7867 \\
        \bottomrule
    \end{tabular}
\end{table}

The initial learning rate was set to \(10^{-3}\) based on standard recommendations for Adam-type optimizers in PINNs applications \cite{raissi2019physics}. We employed a StepLR scheduler, reducing the learning rate by a factor of \(0.9\) every 1000 epochs for the PINNs, and by \(0.5\) every 10,000 epochs for the PI-DeepONet. This scheduler was chosen after a coarse scan over constant learning rates (\(10^{-2}\), \(10^{-3}\), \(10^{-4}\)); the \(10^{-3}\) value with step decay provided faster and more stable loss reduction without divergence. A constant higher rate led to oscillatory loss, while a constant lower rate resulted in extremely slow convergence. The specific decay schedule was determined by monitoring the loss plateau behavior during preliminary runs, a practice consistent with prior inverse PINNs studies \cite {liu2025asymptotic}, \cite{wang2024neural}.

For the PINNs, we used a feedforward network with 3 hidden layers of 64 neurons each and Tanh activation. For the PI-DeepONet, both branch and trunk nets employed 3 hidden layers of 128 neurons each, with a shared output dimension \(p=64\). These architectures were selected through a limited grid search over layer counts (2–5), layer widths (32–256), and output dimensions \(p\) (32–128). The chosen configurations gave the best trade-off between expressivity and training stability, as measured by final validation L2 error and convergence speed. Larger networks (e.g., 5 layers or 256 neurons) did not improve accuracy but increased training time and risk of overfitting, while smaller networks (2 layers or 32 neurons) led to underfitting with visibly higher reconstruction errors.

\bibliographystyle{plain}
\bibliography{reference} 
\end{document}